\documentclass[12pt]{article}

\usepackage{amsmath,amssymb,amsfonts,amsthm}
\usepackage{graphics}
\usepackage{epsfig}
 \usepackage{amsmath}
\usepackage{amsfonts}

\usepackage{geometry}
\font\elevensf=cmss10 scaled\magstephalf

\newtheorem{theorem} {{\elevensf THEOREM}}[section]
\newtheorem{proposition} {{\elevensf PROPOSITION}}[section]

\newtheorem{example} {{\elevensf EXAMPLE}}[section]

\newtheorem{remark} {{\elevensf REMARK}}[section]

\renewcommand\qed{$\blacksquare$}

\def\CC{{\rm \kern.24em \vrule width.02em height1.4ex depth-.05ex \kern-.26emC}}

\def\TagOnRight

\def\AA{{it I} \hskip-3pt{\tt A}}

\def\QQ{\rlap {\raise 0.4ex \hbox{$\scriptscriptstyle |$}} {\hskip -0.1em Q}}

\catcode`\@=11

\newcommand{\lb}{\left(}
\newcommand{\rb}{\right)}

\def\theequation{\@arabic{\c@section}.\@arabic{\c@equation}}

\begin{document}

\baselineskip 14pt
\parindent.4in
\catcode`\@=11 

\begin{center}
{\Huge \bf Dispersion tensor and its unique minimizer in 
Hashin-Shtrikman micro-structures}\\[5mm]
{\bf \textbf{Loredana B\u{a}lilescu, Carlos Conca, Tuhin Ghosh, Jorge San Mart{\'{\i}}n and
\\  Muthusamy Vanninathan}} 
\end{center} 

\begin{abstract}
\noindent
In this paper, we introduce the macroscopic quantity, namely the dispersion tensor
or the \textit{Burnett coefficient}s in the class of generalized Hashin-Shtrikman micro-structures
\cite[page no. 281]{T}. In the case of two-phase materials associated with the periodic 
Hashin-Shtrikman structures, we settle the issue that the dispersion tensor 
has an unique minimizer, which is so called Apollonian-Hashin-Shtrikman micro-structure.
\end{abstract}

\section{Introduction}\label{intro}
\setcounter{equation}{0} 

The aim of this paper is to study the higher order approximation 
in the class of elliptic boundary value problems, where the heterogeneous 
media is governed by the well known generalized Hashin-Shtrikman 
micro-structures \cite{JKO}, \cite[page no. 281]{T}. To first order, we approximate 
the medium by the associated homogenized medium ``$q$''. 
First of all, while the macro tensor ``$q$'' has been introduced for arbitrary
micro-structures, the next order macro tensor, in particular, the fourth order 
dispersion tensor ``$d$'' is introduced so far only for periodic structures 
\cite{COVB,CMSV2,CMSV,CMSV3}. Thanks to the spectral approach to the homogenization 
problem using Bloch waves \cite{CV} which naturally leads to other 
macroscopic quantities apart from the ``$q$''.
Following our previous work \cite{TV1}, here we introduce the idea of higher 
order approximation among Hashin-Shtrikman structures, in particular, we define 
the dispersion tensor and denote it by ``$d_{HS}$''.
Note that, following by the its construction, Hashin-Shtrikman structures 
provide examples of non-periodic structures. More precisely, the inhomogeneous 
coefficients are invariant in a certain way of both translation and dilation of 
the medium with incorporating a family of small scales $\{\varepsilon_p>0\}$. 
Periodic micro-structure incorporates uniform translation and uniform dilation
with respect to only one scale $\varepsilon$, whereas Hashin-Shtrikman 
micro-structures incorporates non-uniform translation and dilation with a family 
of scales $\{\varepsilon_p\}$. It is an open problem to introduce ``$d$'' for more 
general non-periodic structures. One general result concerning ``$d$'' is that it 
has a sign irrespective of the underlying micro-structure and it is negative in 
contrast to ``$q$'' which is positive (see \cite{COVB}).\\
\begin{figure}
 \begin{center}
  \includegraphics[width=14cm]{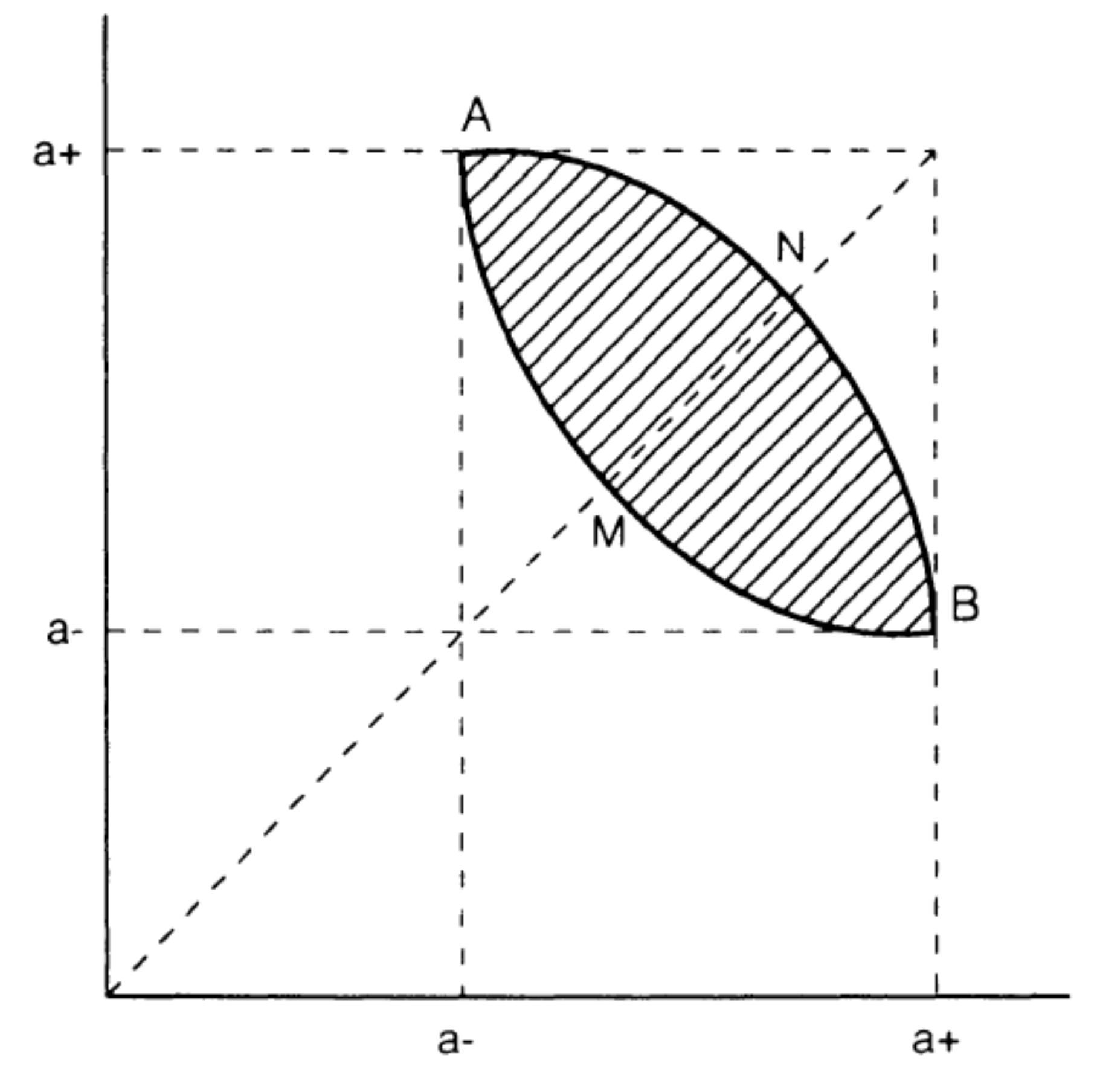}
  \vspace*{-12ex}\hspace*{-26ex}
  \caption{Murat--Tartar Bounds.}
 \end{center}
\end{figure}
 Let us mention the significance of studying of such dispersion tensor in
Hashin-Shtrikman structures. Motivated by Optimal Design Problems (ODP) (see \cite{A}),
one wishes to characterize the first macro tensor ``$q$'' for $2$-phase 
mixtures taken in a known proportion. As we know in the problem of characterizing 
conductivities, the crucial dependence comes from the way they are getting mixed
(micro-structures). Above task is carried out in a
well-known theorem in Homogenization Theory \cite{MT2,MT1}. 
Geometrically, conductivities of all possible mixtures lie in a convex 
lens-shaped region bounded by an upper hyperbola and a lower hyperbola. 
Analytically, conductivities of mixtures satisfy certain inequalities. 
There are two aspects in the theorem: first one is to
find  ``$q$'' corresponding to a given class of micro-structures and the second
one is to construct underlying micro-structures associated to a given
conductivity. Both laminates ($(A,B)$ in Figure 1) and Hashin-Shtrikman 
structures ($(M,N)$ in Figure 1) appear among others as
extremal structures in this characterization in the sense that they are on
the boundary of the above region. No wonder therefore that these structures have
played important role in the construction of the above set.
The goal is to be able to carry out a similar programme for the second
macro tensor ``$d$'' which is known as Bloch dispersion tensor. The name is due
to the fact that it appears in well-posed macro dispersive models
approximating the acoustic equation in periodic media \cite{ABM}. Thus, bounds on
``$d$'' will be useful in ODP in which one wishes to manipulate
micro-structures to have a desired dispersion coefficient. One is far away
from realizing such a goal.\\
First attempt in this endeavor is to see the 
modification that ``$d$'' brings to the phase diagram of ``$q$''. 
A current activity in this attempt is to study the variation of ``$d$'' 
on extremal structures of ``$q$'' and to characterize the range of its values 
on them. Recall that ``$q$'' remains constant when
restricted to such structures, whereas ``$d$'' varies on them. In such a
situation, it is particularly interesting to know its range of values on
them and to know the extremal structure chosen by ``$d$'' to realize its
maximum and minimum values. Such a study was completed in the case of
laminates in \cite{CMSV3}. Some surprises were found in our experience 
with one dimension case \cite{CMSV} and laminates. For instance, ``$d$'' 
picks up a unique element from such laminates at its maximum and minimum. 
Its maximum value is zero and the maximizer is 
a unique relaxed/generalized laminate. However, at its minimum value, the
minimizer is unique, but it is a classical structure and not a relaxed one.\\
Having studied periodic laminates, the next natural step is to treat
periodic Hashin-Shtrikman structures (PHS) and this is what we plan to do in our
present study. Just like molecules are made out of interacting atoms
inside, PHS consist of Hashin-Shtrikman structures inside the periodic cell $Y$.
Our idea is to study these ``atoms'' first. Of course, the major difficulty in this is
that they do not form a periodic structure in the usual sense. However,
there is still the invariance by the action of translation and dilation
groups but the action is now different.
One of the advantages of this work is that it enables us resolve one of
the conjectures (Section \ref{conj} below) regarding the optimal micro-structure, 
when we consider the macro dispersion coefficient on the class periodic 
Hashin-Shtrikman structures.

\section{Preliminaries}
\setcounter{equation}{0}
In the beginning, we remark that the
summation with respect to the repeated indices is understood throughout this paper. Let us start with the known periodic case.
\subsection{Dispersion tensor and periodic structures}
We consider the operator
\begin{equation*} 
\mathcal{A}_Y \equiv -\frac{\partial}{\partial y_k}\lb a^Y_{kl}(y)\frac{\partial}{\partial y_l} \rb ,\quad y\in\mathbb{R}^N, 
\end{equation*}
where the coefficient matrix  $A_Y(y) = [a^Y_{kl}(y)]$ defined on $Y$ a.e. with $Y =[0,1]^N$ is known as the periodic cell
and $A_Y\in\mathcal{M}(\alpha,\beta;Y)$ for some $0<\alpha <\beta$, i.e. 
\begin{equation*}a^Y_{kl}=a^Y_{lk}\hspace{5pt}\forall k,l \mbox{ and }\ (A_Y(y)\xi,\xi) \geq \alpha|\xi|^2,\ \ |A_Y(y)\xi|\leq \beta|\xi|\mbox{ for any } \xi \in \mathbb{R}^N,\mbox{ a.e. on }Y.\end{equation*}
For each $\varepsilon >0 $, we consider the $\varepsilon$-periodic elliptic operator
\begin{equation*} 
\mathcal{A}^{\varepsilon}_Y \equiv -\frac{\partial}{\partial x_k}\Big( a^Y_{kl}\Big(\frac{x}{\varepsilon}\Big)\frac{\partial}{\partial x_l} \Big) ,\quad x\in\mathbb{R}^N, 
\end{equation*}
where $x$ (slow variable) and $y$ (fast variable) are related by $y = \frac{x}{\varepsilon}$.\\
We now define the Bloch waves $\psi_Y$ associated with the operator $\mathcal{A}_Y$.
Let us consider the following spectral problem parametrized by $\eta\in\mathbb{R}^N$:
Find $\lambda_Y=\lambda_Y(\eta)\in \mathbb{R}$ and $\psi_Y=\psi_Y(y;\eta)$ (not zero) such that
\begin{align*}
\mathcal{A}_Y\psi_Y(\cdot;\eta)&= \lambda_Y(\eta)\psi_Y(\cdot;\eta) \quad\mbox{in }\mathbb{R}^N, 
\ \ \psi_Y(\cdot;\eta) \mbox{ is }(\eta;Y)\mbox{--periodic, i.e. }\\
\psi_Y(y+2\pi m;\eta) &= e^{2\pi i m\cdot\eta}\psi_Y(y;\eta) \quad \forall m \in \mathbb{Z}^N, \, y \in \mathbb{R}^N.
\end{align*}
Next, by  Floquet theory, we define $\varphi_Y(y;\eta)=e^{iy\cdot\eta}\psi_Y(y;\eta)$ 
to rewrite the above spectral problem as follows:
\begin{equation}\label{s1}
\mathcal{A}_Y(\eta)\varphi_Y= \lambda_Y(\eta)\varphi_Y \quad\mbox{in }\mathbb{R}^N,\ \ \ \varphi_Y \mbox{ is }Y\mbox{--periodic.}
\end{equation}
Here the operator $\mathcal{A}_Y(\eta)$ is called the translated operator and is defined by
$$\mathcal{A}_Y(\eta)= e^{-iy\cdot\eta}\mathcal{A}_Ye^{iy\cdot\eta} = -\Big(\frac{\partial}{\partial y_k} + i\eta_k\Big)\Big[a^Y_{kl}(y)\Big(\frac{\partial}{\partial y_l} + i\eta_l\Big)\Big].$$
It is well known that for $\eta\in Y^{\prime}= [-\frac{1}{2},\frac{1}{2}[^N$ the dual torus, the above spectral
problem \eqref{s1} admits a discrete sequence of eigenvalues and their eigenfunctions
referred to as Bloch waves introduced above enable us to describe the spectral resolution 
of $\mathcal{A}_Y$ an unbounded self-adjoint operator in $L^2(\mathbb{R}^N)$ in the 
orthogonal basis ͕$\{e^{iy\cdot\eta}\varphi_{Y,m}(y;\eta)| m\geq 1, \eta\in Y^{\prime}\}$.\\ 
To obtain the spectral resolution of $\mathcal{A}^{\varepsilon}_Y$, 
we introduce Bloch waves at the $\varepsilon$-scale as
$$\lambda_{Y,m}^{\varepsilon}(\xi)= \varepsilon^{-2}\lambda_{Y,m}(\eta),\quad \psi_{Y,m}^{\varepsilon}(x;\xi)=\psi_{Y,m}(y;\eta),\quad \varphi_{Y,m}^{\varepsilon}(x;\xi)=\varphi_{Y,m}(y;\eta),$$
where the variables $(x, \xi)$ and $(y, \eta)$ are related by $y = \frac{x}{\varepsilon}$ and $\eta =\varepsilon\xi$.
Observe that $\varphi^{\varepsilon}_{Y,m}(x;\xi)$ is $\varepsilon Y$--periodic (in $x$) and $\varepsilon^{-1}Y^{\prime}$--periodic 
with respect to $\xi$. In the same manner, $\psi^{\varepsilon}_{Y,m}(\cdot; \xi)$ is $(\varepsilon\xi; \varepsilon Y )$--periodic. 
The dual cell at $\varepsilon$-scale, where
$\xi$ varies, is $\varepsilon^{-1}Y^{\prime}$.\\
  
We consider a sequence $u^{\varepsilon} \in H^1(\mathbb{R}^N)$ satisfying 
\begin{equation}\label{ed1}
\mathcal{A}^{\varepsilon}_Yu^{\varepsilon} = f \quad\mbox{ in }\mathbb{R}^N,
\end{equation} 
with the fact $u^{\varepsilon}\rightharpoonup u$ in $H^{1}(\mathbb{R}^N)$ 
weak and $u^{\varepsilon}\rightarrow u $ in $L^{2}(\mathbb{R}^N)$ strong.\\
The homogenization problem consists of passing to the limit in \eqref{ed1}, 
as $\varepsilon \rightarrow 0$ and we get the homogenized equation 
satisfied by $u$, namely
\begin{equation*} \mathcal{A}^{*}_Yu =\ -\frac{\partial}{\partial x_k}\Big( q_{kl}\frac{\partial u}{\partial x_l} \Big) = f \quad\mbox{in }\mathbb{R}^N, \end{equation*}
where $A^{*}_Y=[q_{kl}]$ is the constant homogenized matrix (see \cite{A}).\\

Simple relation linking $A^{*}_Y$ with Bloch waves is the following: $q_{kl}=\frac{1}{2}D^2_{kl}\lambda_{Y,1}(0)$ (see \cite{CV}).
At this point, it is appropriate to recall that derivatives of the first eigenvalue
and eigenfunction at $\eta=0$ exist, thanks to the regularity property established in \cite{COV,CV}.
\begin{proposition}[Regularity of the ground state \cite{COV,CV}]\label{based}
Under the periodic assumption on the matrix $A_Y\in \mathcal{M}(\alpha,\beta; Y)$, there  exists $\delta > 0$
such that the first eigenvalue $\lambda_{Y,1}(\eta)$ is an analytic function on 
$B_{\delta}(0) = \{\eta\in\mathbb{R}^N \ |\ |\eta| < \delta\}$ and
there is a choice of the first eigenvector $\varphi_{Y,1}(y;\eta)$ satisfying
\begin{equation*}
\eta \mapsto \varphi_{Y,1}(\cdot;\eta) \in H^1_{\#}(Y)\mbox{ is analytic on } B_{\delta}\mbox{ and }\ \varphi_{Y,1}(y;0)= |Y|^{-1/2}.
\end{equation*}
Moreover, we have the following relations:
\begin{equation*}\begin{aligned}
& \lambda_{Y,1}(0) = 0,\quad D_k \lambda_{Y,1}(0) = \frac{\partial\lambda_{Y,1}}{\partial \eta_k}(0) = 0 \quad\forall k = 1,..,N.\\
& \varphi_{Y,1}(\cdot,0)= |Y|^{-1/2},\quad D_k\varphi_{Y,1}(\cdot,0) = i|Y|^{-1/2}\chi_k(y),\\ 
&\frac{1}{2}D^2_{kl}\lambda_{Y,1}(0) = \frac{1}{2}\frac{\partial^2\lambda_{Y,1}}{\partial \eta_k\partial\eta_l}(0) = q_{kl} \quad\forall k,l =1,\ldots,N,
\end{aligned}
\end{equation*}
where the last expression is considered as the Bloch spectral representation of the homogenized tensor,
which are essentially defined as
\begin{equation}\label{ed2} 
 \frac{1}{2}D^2_{kl}\lambda_{Y,1}(0) = q_{kl}=\  \frac{1}{|Y|}\int_{Y} A_Y(\nabla \chi_k + e_k)\cdot(\nabla \chi_l + e_l) dy, 
\end{equation}
 for each unit vector $e_{k}$ and the functions $\chi_k\in H^1_{\#}(Y)$ solving the following conductivity problem in the periodic unit cell: 
\begin{equation*}
-div_y\lb A_Y(y)(\nabla_y\chi_k(y)+e_k)\rb = 0 \quad\mbox{in }Y,\quad y \longmapsto\chi_k(y) \quad\mbox{is $Y-$periodic.}
\end{equation*}
Moreover, all odd order derivatives of $\lambda_{Y,1} $ at $\eta=0$ are zero, i.e. 
\begin{equation*}D^{q}\lambda_{Y,1}(0)=0 \quad\forall q\in \mathbb{Z}^N_{+}, \, |q|\mbox{ odd.} 
\end{equation*}
Additionally, all even order derivatives need not to be zero and can be calculated in a systematic way.\\
The fourth order derivative of $\lambda_{Y,1}$ at $\eta=0$ is non-zero and 
known as the \textbf{Burnett coefficient} or the \textbf{dispersion tensor} $d_{Y}$ of the medium,
\begin{equation}\label{dl22}
 \frac{1}{4!}\frac{\partial^4\lambda_{Y,1}(0)}{\partial\eta_k\partial\eta_l\partial\eta_m\eta_n}\eta_k\eta_l\eta_m\eta_n = d_{klmn}\eta_k\eta_l\eta_m\eta_n =\ d_{Y}\eta^4,
\end{equation}
which is essentially a non-positive definite fourth order tensor and can be expressed as follows:
let us call $D^2_{kl}\varphi_{Y,1}(\cdot;0) = |Y|^{-1/2}\chi_{kl}$ and define 
\begin{equation}\begin{aligned}\label{dl19}
C_Y &=\eta_n C^Y_n \quad\mbox{ with }\quad C^Y_n(\varphi) = - a^Y_{nj}(y)\frac{\partial \varphi}{\partial y_j} - \frac{\partial}{\partial y_j}(a^Y_{nj}(y)\varphi), \\
X^{(1)}_Y &= \eta_n\chi_{n},\quad X^{(2)}_Y= \eta_k\eta_n\chi_{kn},\quad  \widetilde{A_Y}= \eta_k\eta_n a^Y_{kn},\quad \widetilde{A^{*}_Y}=\eta_k\eta_n q_{kn}
\end{aligned}\end{equation}
satisfying
\begin{equation*}
-div(A_Y\nabla X^{(1)}_Y) = \eta_k\frac{\partial a^Y_{kl}}{\partial y_l} \mbox{ in }Y, \quad X^{(1)}_Y\in H^1_{\#}(Y) \mbox{ with }\int_Y X^{(1)}_Y dy = 0
\end{equation*} 
and  
\begin{equation}\label{ll9}
-div(A_Y\nabla X^{(2)}_Y) = (\widetilde{A_Y}-\widetilde{A^{*}_Y}) - C_Y X^{(1)}_Y \mbox{ in }Y, \quad X^{(2)}_Y\in H^1_{\#}(Y) \mbox{ with }\int_{Y} X^{(2)}_Ydy =0.
\end{equation}
Then, by summation, one has the following expression of the dispersion tensor: 
\begin{equation}\label{dy}
d_{Y}\eta^4 = \ -\frac{1}{|Y|}\int_{Y} \mathcal{A}_Y\Big( X^{(2)}_Y -\frac{(X^{(1)}_Y)^2}{2}\Big)\cdot\Big( X^{(2)}_Y -\frac{(X^{(1)}_Y)^2}{2}\Big)\leq 0.
\end{equation}
\end{proposition}
\begin{remark}
In order to see the role of the dispersion tensor $d_Y$ arises in wave propagation problems, let us consider the wave propagation problem in  periodic
structure governed by the operator $\partial_{tt} + \mathcal{A}^{\varepsilon}_Y$
with appropriate initial conditions. As we see,
we have\begin{align*}
\lambda_{Y,1}^{\varepsilon}(\xi) &\approx\ \frac{1}{2!} \lambda_{Y,1}^{(2)}(0)\xi^2 \quad\mbox{if }\varepsilon^2|\xi|^4 \mbox{ is small,}\\
\lambda_{Y,1}^{\varepsilon}(\xi) &\approx\ \frac{1}{2!}\lambda_{Y,1}^{(2)}(0)\xi^2 + \frac{1}{4!}\varepsilon^2\lambda_{Y,1}^{(4)}(0)\xi^4 \quad\mbox{if }\varepsilon^4|\xi|^6 \mbox{ is small}.
\end{align*}
Thus, if we consider short waves of low energy with wave number satisfying 
$\varepsilon^2|\xi|^4 = \mathcal{O}(1)$ and $\varepsilon^4|\xi|^6 = o(1)$,
then a simplified description is obtained with the operator 
$\partial_{tt} + \mathcal{A}^{*}_Y + \varepsilon^2D_Y$,
where $D_Y$ is the fourth-order operator whose symbol is 
$\frac{1}{4!}\frac{\partial^4\lambda_{Y,1}(0)}{\partial\eta_k\partial\eta_l\partial\eta_m\eta_n}\xi_k\xi_l\xi_m\xi_n$.
\end{remark}

\subsection{Survey of Bloch waves, Bloch eigenvalues and eigenvector 
in Hashin-Shtrikman structure}
In this part, we recall our recent work \cite{TV1} of introducing Bloch waves and associated 
Bloch spectral analysis in the class of generalized Hashin-Shtrikman micro-structures concerning the homogenization result.
\subsection*{Hashin-Shtrikman micro-structures}
We follow \cite[page no. 281]{T} in this sequel. 
Let $\omega \subset \mathbb{R}^N$ be a bounded open subset with Lipschitz boundary.
Let $A_\omega(y)=[a^{\omega}_{kl}(y)]_{1\leq k,l\leq N} \in \mathcal{M}(\alpha,\beta;\omega)$  
be such that after extending $A_\omega$ by $A_\omega(y) = M$ for $x \in \mathbb{R}^N\smallsetminus \omega$, where $M \in L_{+}(\mathbb{R}^N ; \mathbb{R}^N)$
(i.e. $M = [m_{kl}]_{1\leq k,l \leq N}$ is a constant positive definite $N\times N$ matrix), 
if for each $\lambda \in \mathbb{R}^N$ there exists $w_{\lambda}\in  H^{1}_{loc}(\mathbb{R}^N)$ satisfying
\begin{equation}\label{hsw}
- div (A_{\omega}(y)\nabla w_{\lambda}(y)) = 0 \quad\mbox{in }\mathbb{R}^N,\quad  w_{\lambda}(y) = (\lambda, y) \quad\mbox{in }\mathbb{R}^N \smallsetminus \omega,
\end{equation}
then $A$ is said to be \textit{equivalent} to $M$.\\
\\
Then one uses a sequence of Vitali coverings of $\Omega$ by reduced copies of $\omega,$
\begin{equation}\label{ll1}  meas\big(\Omega \smallsetminus \underset{p\in K}{\cup} (\varepsilon_{p,n}\omega + y^{p,n})\big) = 0, \mbox{ with } \kappa_n = \underset{p\in K}{sup}\hspace{2pt} \varepsilon_{p,n}\rightarrow 0,
\end{equation}
for a finite or countable $K$. These define the micro-structures in $A^{n}_\omega$. One defines for almost everywhere $x\in \Omega$, 
\begin{equation}\label{hs} 
A^{n}_\omega(x) = A_\omega\Big(\frac{x - y^{p,n}}{\varepsilon_{p,n}}\Big)\ \ \mbox{ in } \varepsilon_{p,n}\omega + y^{p,n},\quad p\in K,
\end{equation}
which makes sense since, for each $n$, the sets $\varepsilon_{p,n}\omega + y^{p,n}, p\in K$ are disjoint.
The above construction \eqref{hs} represents the so called \textbf{Hashin-Shtrikman} micro-structures.\\
\\
Following that, one defines $v^{n} \in H^1(\Omega)$ by
\begin{equation}\label{uB}
v^{n}(x) = \varepsilon_{p,n}w_{\lambda}\Big(\frac{x-y^{p,n}}{\varepsilon_{p,n}}\Big)+ (\lambda, y^{p,n})\quad\mbox{ in  }\varepsilon_{p,n}\omega + y^{p,n}.
\end{equation}
Then one has the following properties (see \cite[Page no. 283]{T}):
\begin{equation}\begin{aligned}\label{uC}
 v^{n}(x) &\rightharpoonup (\lambda,x) \mbox{ weakly in }H^1(\Omega;\mathbb{R}^N),\\
 A^{n}_\omega\nabla v^{n}(x) &\rightharpoonup M\lambda \mbox{ weakly in } L^2(\Omega;\mathbb{R}^N),\\
 - div( A^{n}_\omega(x)\nabla v^{n}(x)) &=\ 0 \ \mbox{ in }\Omega.
\end{aligned}\end{equation}
So, by the definition of $H$-convergence (see \cite[Page no. 82]{T}), one has the following convergence of the entire sequence
\begin{equation} A^{n}_\omega \xrightarrow{H-\mbox{converges }} M, 
\end{equation}
where $M\in L_{+}(\mathbb{R}^N,\mathbb{R}^N)$ is a positive definite matrix equivalent to $A$.\\
\\
We have the following integral representation similar to \eqref{ed2}:
\begin{equation}\begin{aligned}\label{dl12}
Me_k\cdot e_l=\ m_{kl} &=\ \frac{1}{|\omega|}\int_{\omega} A_\omega(y)\nabla w_{e_k}(y)\cdot\nabla w_{e_l}(y)\ dy \\
                       &=\ \frac{1}{|\omega|}\int_{\omega} A_\omega(y)\nabla w_{e_k}(y)\cdot e_l\ dy, 
\end{aligned}\end{equation}
where $w_{e_k}, w_{e_l}$ are the solution of \eqref{hsw} for $\lambda= e_k$ and $\lambda=e_l$, respectively.
\hfill\qed
\begin{example}
[Spherical Inclusions in two-phase medium] \label{si}
If $\omega= B(0,1)=\{ y\ | \ |y|\leq 1\} $ and
\begin{equation*}
     A_\omega(y)=a_B(r)I=
        \left\{
     \begin{array}{ll} 
        \alpha I \quad\mbox{if } |y| \leq R,\\[1ex]
        \beta I \quad\mbox{if }  R < |y| \leq 1,
        \end{array}
        \right.
\end{equation*}
$\alpha$ and $\beta$ are known as core and coating, respectively. Then $A_\omega$ is equivalent to $\gamma I$, where $\gamma$ satisfies 
\begin{equation*}
\frac{\gamma - \beta}{\gamma + (N-1)\beta} = \theta \frac{\alpha - \beta}{\alpha + (N-1)\beta},\ \mbox{ with }\theta= R^N .
\end{equation*} 
\end{example}
\begin{example}[Elliptical Inclusions in two-phase medium]\label{ei}
For $m_1,\ldots,m_N\in \mathbb{R}$ and $\rho + m_j >0 $ for $j=1,\ldots,N$, the 
family of confocal ellipsoids $S_\rho$ of equation
$$ \sum_{j=1}^N \frac{y^2_j}{\rho + m_j} = 1$$ defines implicitly a
real function $\rho$, outside a possibly degenerate ellipsoid in
a subspace of dimension $<$ $N$.\\
Now, if we consider $\omega=E_{\rho_2+m_1,\ldots,\rho_2+m_N}= \Big\{ y\ | \ \sum\limits_{j=1}^N \frac{y^2_j}{\rho_2 + m_j} \leq 1\Big\},$ 
with $\rho_2 + \underset{j}{min}\ m_j >0 $  and  
\begin{equation*}
 A_\omega(y)= a_E(\rho)I =
 \left\{ 
 \begin{array}{ll}
 \alpha I \quad\mbox{if } \rho \leq \rho_1,\\[1ex]
 \beta I \quad\mbox{if }  \rho_1 < \rho \leq \rho_2,
 \end{array}
 \right. 
\end{equation*}
then $A_\omega$ is equivalent to a constant diagonal matrix $\Gamma= [\gamma_{jj}]_{1\leq j\leq N}$ satisfying
$$ \sum_{j=1}^N \frac{1}{\beta - \gamma_{jj}} =\ \frac{(1-\theta)\alpha + (N+\theta-1)\beta}{\theta \beta(\beta-\alpha)},\ \mbox{ with }\theta = \underset{j}{\Pi}\ \sqrt{\frac{\rho_1 + m_j}{\rho_2 + m_j}}.$$ 
\end{example}
\hfill\qed
\subsection*{Bloch waves, Bloch eigenvalues and eigenvectors associated with the Hashin-Shtrikman structures}
Let $\omega \subset \mathbb{R}^N$ be a bounded open domain with Lipschitz boundary and $A_\omega(y)=[a^\omega_{kl}(y)]_{1\leq k,l\leq N}\in \mathcal{M}(\alpha,\beta,\omega)$. 
We consider the following spectral problem parameterized by $\eta \in \mathbb{R}^N$: 
Find $\lambda_\omega := \lambda_\omega(\eta) \in \mathbb{C}$ and $\varphi_\omega := \varphi_\omega(y; \eta)$ (not identically zero) such that
\begin{equation}\begin{aligned}\label{NE1}
\mathcal{A}_\omega(\eta)\varphi_\omega(y;\eta)=\ -\Big(\frac{\partial}{\partial y_k} + i\eta_k\Big)&\Big[a^\omega_{kl}(y)\Big(\frac{\partial}{\partial y_l} + i\eta_l\Big)\Big]\varphi_\omega(y;\eta) = \lambda_\omega(\eta)\varphi_\omega(y;\eta) \mbox{ in }\omega, \\
\varphi_\omega(y;\eta)\mbox{ is constant on }\partial\omega &\mbox{ and }
\int_{\partial\omega} a^\omega_{kl}(y)\Big(\frac{\partial}{\partial y_l} + i\eta_l\Big)\varphi_\omega(y;\eta)\nu_k\ d\sigma = 0, 
\end{aligned}\end{equation}
where $\nu$ is the outer normal vector on the boundary and $d\sigma$ is the surface measure on $\partial\omega$.
\\
\noindent
We introduce the state spaces of this above spectral problem:
\begin{align*}
L^2_c(\omega) &=  \ \{\varphi\in L^2_{loc}(\mathbb{R}^N) \  | \ \varphi \mbox{ is constant in } \mathbb{R}^N \smallsetminus\omega\},\\
H^1_c(\omega) &=\ \{\varphi\in H^1_{loc}(\mathbb{R}^N) \  | \ \varphi \mbox{ is constant in } \mathbb{R}^N \smallsetminus\omega\}\\
              &= \ \{\varphi\in H^1(\omega) \  | \ \varphi|_{\partial\omega} = \mbox{ constant}\}.
\end{align*}
Here, “$c$” is a floating constant depending on the element under consideration.
$L^2_c(\omega)$ and $H^1_c(\omega)$ are proper subspace of $L^2(\omega)$ and  $H^1(\omega)$ respectively, and they inherit
the subspace norm-topology of the parent space. 

Prior to that, we have this following result establishing the existence of the Bloch eigenelements.
\begin{proposition}[Existence result \cite{TV1}]\label{dl1}
Fix $\eta \in \mathbb{R}^N$. Then, there exist a sequence of eigenvalues $\{ \lambda_{\omega,m}(\eta); m \in \mathbb{N} \}$
and its corresponding eigenvectors $\{ \varphi_{\omega,m} (y; \eta)\in H^1_c(\omega), m \in \mathbb{N} \}$ such that
\begin{align*}
&(i) \ \mathcal{A}_\omega(\eta)\varphi_{\omega,m}(y;\eta) = \lambda_{\omega,m}(\eta)\varphi_{\omega,m}(y;\eta) \quad \forall m \in \mathbb{N}.\\ 
&(ii) \ 0 \leq \lambda_{\omega,1}(\eta) \leq \lambda_{\omega,2} (\eta) \leq \ldots \rightarrow \infty; \ \mbox{each eigenvalue is of finite multiplicity. }\\
&(iii)\ \{\varphi_{\omega,m}(\cdot;\eta);\ m \in  \mathbb{N} \}\ \mbox{ is an orthonormal basis for }L^2_c(\omega).\\
&(iv)\ \mbox{ For }\phi \mbox{ in the domain of } \mathcal{A}_\omega(\eta), \mbox{ we have }\\
&\hspace{3cm}\mathcal{A}_\omega(\eta)\phi(y) = \sum_{m=1}^{\infty} \lambda_{\omega,m} (\eta)\big(\phi, \varphi_{\omega,m} (\cdot;\eta)\big)\varphi_{\omega,m} (y;\eta).
\end{align*}
\end{proposition}
\noindent
As the eigen-branch emanating from the first eigenvalue plays the key role,
we concentrate only for $m=1$ to have the following regularity properties.
\begin{proposition}[Regularity of the ground state \cite{TV1}]
Let $\lambda_{\omega,1}(\eta),\varphi_{\omega,1}(\cdot;\eta)$ be the first eigenvalue and the first eigenvector of the spectral problem defined in \eqref{NE1}. 
Then, there exists a neighborhood $\omega^{\prime}$ around zero such that
\begin{equation*}
\eta \longmapsto \lb  \lambda_{\omega,1}(\eta), \varphi_{\omega,1}(\cdot;\eta)\rb \in \mathbb{C}\times H^1_c(\omega)\mbox{ is analytic on }\omega^{\prime}.
\end{equation*}
At $\eta=0$, $\lambda_{\omega,1}(0)$ is simple. There is a choice of the first eigenvector $\varphi_{\omega,1}(y;\eta)$ satisfying
\begin{equation*}
\varphi_{\omega,1}(y;\eta) =\ \frac{1}{|\omega|^{1/2}} \ \forall y\in\partial\omega\mbox{ and }\forall\eta\in \omega^{\prime} .
\end{equation*}
Moreover, we have the following relations:\begin{equation}\begin{aligned}\label{lZ}
& \lambda_{\omega,1}(0) = 0,\quad D_k \lambda_{\omega,1}(0)=\frac{\partial\lambda_{\omega,1}}{\partial \eta_k}(0)= 0 \quad\forall k = 1,\ldots,N,\\
& \varphi_{\omega,1}(\cdot;0)= |\omega|^{-1/2},\quad D_k\varphi_{\omega,1}(y,0) = i|\omega|^{-1/2}(w_{e_k}(y)-y_k),\\ 
&\frac{1}{2}D^2_{kl}\lambda_{\omega,1}(0) = \frac{1}{2}\frac{\partial^2\lambda_{\omega,1}}{\partial \eta_k\partial\eta_l}(0) = m_{kl} \quad\forall k,l =1,\ldots,N,
\end{aligned}\end{equation}
where the last expression is considered as a Bloch spectral representation of the homogenized tensor.
\hfill\qed\end{proposition}
\noindent
Moreover, all odd order derivatives of $\lambda_{\omega,1} $ at $\eta=0$ are zero, i.e. 
\begin{equation}D^{q}\lambda_{\omega,1}(0)=0 \quad\forall q\in \mathbb{Z}^N_{+},\ |q| \mbox{ odd}. \end{equation}
In particular, the third order derivative is zero. However,  we are interested in the further next order approximation 
by calculating the fourth order derivatives of $\lambda_{\omega,1}(0)$, i.e.
$D^4_{klmn}\lambda_{\omega,1}(0)$, which is in general a non-positive definite tensor and can be defined as follows:
the second order derivative of the eigenvector $D^2_{kl}\varphi_{\omega,1}(\cdot;0)\in H^1_0(\omega)$ solves
\begin{equation}\begin{aligned}\label{dl17}
&\mathcal{A}D^2_{kl}\varphi_{\omega,1}(y;0)= -( a^\omega_{kl}(y) - m_{kl} )\varphi_{\omega,1}(y;0) - iC_k(D_l(\varphi_{\omega,1}(y;0)) - iC_l(D_k\varphi_{\omega,1}(y;0))\mbox{ in }\omega,\\
&D^2_{kl}\varphi_{\omega,1}(y;0) = 0 \mbox{ on }\partial\omega \mbox{ and} \int_{\partial\omega} A_\omega(y)\nabla_y D^2_{kl}\varphi_{\omega,1}(y;0)\cdot\nu \ d\sigma= 0 .
\end{aligned}\end{equation}
We call $D^2_{kl}\varphi_{\omega,1}(y;0) = |\omega|^{-1/2}w_{kl}(y)$ and let us define
\begin{center}$X^{(1)}_{\omega} =\ \eta_k(w_{e_k}(y)-y_k)\mbox{ and }\ X^{(2)}_{\omega}=\ \eta_k\eta_lw_{kl}\ $ likewise in \eqref{dl19}.\end{center}
%where they satisfy
% \begin{equation}\begin{aligned}
% -div(A\nabla X^{(1)}) &= \eta_k\frac{\partial a_{kl}}{\partial y_l} \mbox{ in }\omega \\
%  X^{(1)}=0\mbox{ on }\partial\omega\quad&\mbox{and}\quad \int_{\partial\omega} A\lb \nabla X^{(1)} + \eta \rb\cdot\nu\ \ d\sigma  = 0.
% \end{aligned}\end{equation} 
% and  
% \begin{equation}\begin{aligned}\label{dl13}
% -div(A\nabla X^{(2)}) &= (\widetilde{A}-\widetilde{M}) - CX^{(1)} \mbox{ in }\omega\\
% X^{(2)}=0\mbox{ on }\partial\omega\quad\mbox{and}&\quad \int_{\partial\omega} A\nabla X^{(2)}\cdot\nu\ \ d\sigma  = 0.
% \end{aligned}\end{equation}
Then, by summation, following \cite[Proposition 3.2]{COVB} it can be shown 
that the following expression defines the fourth order derivative of $\lambda_{\omega,1}(\eta)$ at $\eta=0$:
\begin{equation}\begin{aligned}\label{dl20}
\frac{1}{4!}D^4_{klmn}\lambda_{\omega,1}(0)\eta_k\eta_l\eta_m\eta_n &= -\frac{1}{|\omega|}\int_{\omega} \mathcal{A}\Big( X^{(2)}_{\omega} - \frac{1}{2}(X^{(1)}_{\omega})^2\Big)\cdot \Big( X^{(2)}_{\omega} - \frac{1}{2}(X^{(1)}_{\omega})^2\Big)  dy\\
&\leq 0\ .
\end{aligned}\end{equation}
This tells us that $\lambda_{\omega,1}^{(4)}(\eta)$ at $\eta=0$ is a non-positive definite tensor.
\hfill\qed
\\
\\
Next, we consider a medium in $\Omega$ with Hashin-Shtrikman micro-structures.
Let us introduce the operator $\mathcal{A}^{n}_\omega$ governed with the Hashin-Shtrikman construction: 
\begin{equation}\label{dl5}\mathcal{A}^{n}_\omega = -\frac{\partial}{\partial x_k}\Big(a^{n}_{kl}(x)\frac{\partial}{\partial x_l}\Big)
\ \mbox{ with  }\ a_{kl}^{n}(x) =\ a^\omega_{kl}\Big(\frac{x-y^{p,n}}{\varepsilon_{p,n}}\Big)\ \mbox{ in } \varepsilon_{p,n}\omega + y^{p,n} \mbox{ a.e. on } \Omega, \end{equation}
where $ meas\big(\Omega \smallsetminus \cup_{p\in K} (\varepsilon_{p,n}\omega + y^{p,n})\big) = 0,$ with 
$\kappa_n = \underset{p\in K}{sup}\hspace{2pt} \varepsilon_{p,n}\rightarrow 0$ 
for a finite or countable $K$ and, for each $n$, the sets 
$\varepsilon_{p,n}\omega + y^{p,n},\ p\in K$ are disjoint.\\
\\
We obtain the spectral resolution of $\mathcal{A}^{n}_\omega$ for fixed $n$, in each $\{\varepsilon_{p,n}\omega +y^{p,n}\}_{p\in K}$ domain, in an analogous manner.  
We introduce the following shifted operator 
\begin{equation}\label{dl2} 
(\mathcal{A}^{n,p}_\omega)(\xi)=\ -\Big(\frac{\partial}{\partial x_k} + i\xi_k\Big)\Big( a^\omega_{kl}\Big(\frac{x-y^{p,n}}{\varepsilon_{p,n}}\Big)\Big(\frac{\partial}{\partial x_l} +i\xi_l\Big) \Big),\quad x\in \varepsilon_{p,n}\omega +y^{p,n}.
\end{equation}
By homothecy, for a fixed $n$ and for each $p$, we define the first Bloch eigenvalue $\lambda_{\omega,1}^{n,p}(\xi)$
and the corresponding Bloch mode $\varphi_{\omega,1}^{n,p}(\cdot;\xi)$ for the operator $(\mathcal{A}^{n,p}_\omega)(\xi)$ for $\xi \in \kappa_n^{-1}\omega^{\prime}$ as follows:
\begin{equation}\label{dl3}
\lambda_{\omega,1}^{n,p}(\xi) := \varepsilon_{p,n}^{-2}\lambda_{\omega,1}(\varepsilon_{p,n} \xi),\quad 
\varphi^{n,p}_{\omega,1}(x;\xi):= \varphi_{\omega,1}\Big(\frac{x-y^{p,n}}{\varepsilon_{p,n}}; \varepsilon_{p,n}\xi\Big), \quad x\in\varepsilon_{p,n}\omega + y^{p,n},
\end{equation}
where $\lambda_{\omega,1}(\eta)$ and $\varphi_{\omega,1}(y;\eta)$ are the eigenelements defined in Proposition \ref{dl1}. \\
\\
This leads to define the Bloch transformation in $L^2(\mathbb{R}^N)$ in the following manner:
\begin{proposition}[Bloch transformation \cite{TV1}]
\noindent
\begin{enumerate}
\item For $g \in L^2(\mathbb{R}^N)$, for each $n$, the following limit in $L^2(\kappa_n^{-1}\omega^{\prime})$ space exists:
\begin{equation}\label{btype}
B^{n}_1 g(\xi): = B^{(\varepsilon_{p,n},\ y^{p,n})}_1 g(\xi) := \sum_p\int_{\varepsilon_{p,n}\omega + y^{p,n}} g(x)e^{-ix\cdot\xi}\overline{\varphi_{\omega,1}}\Big(\frac{x-y^{p,n}}{\varepsilon_{p,n}};\varepsilon_{p,n}\xi\Big)dx,
\end{equation}
where, for each $n$,  $ meas\big( \mathbb{R}^N \smallsetminus \underset{p\in K}{\cup} (\varepsilon_{p,n}\omega + y^{p,n})\big) = 0,$ 
with $\kappa_n = \underset{p\in K}{sup}\hspace{2pt} \varepsilon_{p,n}\rightarrow 0$ 
for a finite or countable $K$ and the sets $\varepsilon_{p,n}\omega + y^{p,n},\ p\in K$ are disjoint.\\
\\
The above definition \eqref{btype} is the corresponding first Bloch transformation governed with
Hashin-Shtrikman micro-structures.
\item We have the following Bessel inequality for elements of $L^2(\mathbb{R}^N)$:
\begin{equation}\label{basel}
\int_{\kappa_n^{-1}\omega^{\prime}} |B^{n}_1 g(\xi)|^2 d\xi \leq \mathcal{O}(1) ||g||^2_{L^2(\mathbb{R}^N)}.
\end{equation}
\item For $g\in H^1(\mathbb{R}^N)$, we have
\begin{equation}\label{rln}
B^{n}_1 \lb\mathcal{A}^{n}_{\omega}g(\xi)\rb := \sum_p\int_{\varepsilon_{p,n}\omega + y^{p,n}}\lambda_{\omega,1}^{n,p} g(x)e^{-ix\cdot\xi}\overline{\varphi_{\omega,1}}\Big(\frac{x-y^{p,n}}{\varepsilon_{p,n}};\varepsilon_{p,n}\xi\Big)dx.
\end{equation}
\end{enumerate}
\end{proposition}
\noindent
One has the first Bloch transform is an approximation to the Fourier transform.
\begin{proposition}[First Bloch transform tends to Fourier transform \cite{TV1}]
\noindent
\begin{enumerate}
 \item  If $g_{n} \rightharpoonup g$ in $L^2(\mathbb{R}^N)$ weak, then
$\chi_{\kappa_n^{-1}\omega^{\prime}}(\xi)B_1^{n}g^{n}(\xi) \rightharpoonup \widehat{g}(\xi)$
in $L^2(\mathbb{R}^N)$ weak, provided there is a fixed compact $R$ 
such that support of $ g^n\subseteq R\quad\forall n.$ 
\item  If $g_{n} \rightarrow g$ in $L^2(\mathbb{R}^N)$ strong, then for the subsequence $\varepsilon_{p,n}$,
$\chi_{\kappa_n^{-1}\omega^{\prime}}(\xi)B_1^{n}g^n \rightarrow \widehat{g}(\xi)$
in $L^2_{loc}(\mathbb{R}^N).$
\end{enumerate}
\end{proposition}
\noindent
Using these above tools the following homogenization theorem has been deduced in \cite{TV1}.
\begin{theorem}[Homogenization result \cite{TV1}]
Let us consider $\Omega$ be an open subset of $\mathbb{R}^N$
and consider the operator $\mathcal{A}^{n}_\omega$ introduced in \eqref{dl5} governed with the Hashin-Shtrikman construction, where
the matrix $A_\omega\in \mathcal{M}(\alpha,\beta;\omega)$ is equivalent to $M$ in the sense of \eqref{hsw}.
Let $f \in L^2(\Omega)$ and consider $u^{n} \in H^1_0(\Omega)$ being the unique solution of the boundary value problem 
\begin{equation*} \mathcal{A}^{n}_\omega u^{n} = f \quad\mbox{in }\Omega.\end{equation*}
Then, there exists $u\in H^1_{0}(\Omega)$ such that the sequence $u^{n}$ converges to $u$ in $H^1_{0}(\Omega)$ weak,
with the following convergence of the flux:
\begin{equation*}
\sigma^{n}_{\omega}= A^{n}_{\omega}\nabla u^{n} \rightharpoonup M\nabla u =\sigma_{\omega} \quad\mbox{in } L^2(\Omega)\mbox{ weak. }
\end{equation*}
In particular, the limit $u$ satisfies homogenized equation:
\begin{equation*}\mathcal{A}^{*}_{\omega}u = -\frac{\partial}{\partial x_l}\Big(m_{kl}\frac{\partial}{\partial x_k}u\Big) = f \quad\mbox{in }\Omega.\end{equation*}
\end{theorem}
\hfill
\qed
\\
We end our discussion here concerning with the homogenized matrix. In the next section we will move into defining the dispersion tensor for the Hashin-Shtrikman micro-structures.
\section{Dispersion tensor and Hashin-Shtrikman structures}\label{conj}
\setcounter{equation}{0}
Here we are going to define the dispersion tensor or the \textit{Burnett coefficient,} more precisely, 
the fourth order approximation of the medium governed by the Hashin-Shtrikman micro-structures, while in the previous section we have studied the 
homogenized coefficient as a second order approximation of the medium. \\
\\
Let us consider $\Omega$ be an open subset of $\mathbb{R}^N$. We recall \eqref{dl5} where we have introduced the operator $\mathcal{A}^{n}_\omega$ governed with the Hashin-Shtrikman construction: 
$$\mathcal{A}^{n}_\omega = -\frac{\partial}{\partial x_k}\Big(a^{n}_{kl}(x)\frac{\partial}{\partial x_l}\Big),$$
with  $A^{n}_\omega(x)=\ [a_{kl}^{n}(x)] =\Big[a^\omega_{kl}\Big(\frac{x-y^{p,n}}{\varepsilon_{p,n}}\Big)\Big]\ \mbox{ in } \varepsilon_{p,n}\omega + y^{p,n}, \mbox{ a.e. on } \Omega$,
where $ meas\big(\Omega \smallsetminus \cup_{p\in K} (\varepsilon_{p,n}\omega + y^{p,n})\big) = 0,$ with 
$\kappa_n = \underset{p\in K}{sup}\hspace{2pt} \varepsilon_{p,n}\rightarrow 0$ 
for a finite or countable $K$ and, for each $n$, the sets 
$\varepsilon_{p,n}\omega + y^{p,n},\ p\in K$ are disjoint.\\
\\
Previously, for each $n$, we restricted the operator $\mathcal{A}^{n}_\omega$ in each $\{\varepsilon_{p,n}\omega +y^{p,n}\}_{p\in K}$ to define $\mathcal{A}^{n,p}_\omega$. Then,
by homothecy, we obtained its first Bloch spectral data $(\lambda_{\omega,1}^{n,p},\ \varphi_{\omega,1}^{n,p})$ in \eqref{dl3}. We have
the following Taylor expansion around zero:\begin{align*}
\lambda_{\omega,1}^{n,p}(\xi) &=\ \frac{1}{2!}\frac{\partial^2\lambda_{\omega,1}}{\partial\eta_k\partial\eta_l}(0)\xi_k\xi_l + \varepsilon_{p,n}^2\frac{1}{4!}\frac{\partial^4\lambda_{\omega,1}}{\partial\eta_k\partial\eta_l\partial\eta_m\partial\eta_n}(0)\xi_k\xi_l\xi_m\xi_n + o(\varepsilon_{p,n}^2)\\ 
                                &=\ m_{kl} \xi_k\xi_l + \varepsilon_{p,n}^2\frac{1}{4!}\frac{\partial^4\lambda_{\omega,1}}{\partial\eta_k\partial\eta_l\partial\eta_m\partial\eta_n}(0) \xi_k\xi_l\xi_m\xi_n + o(\varepsilon_{p,n}^2),\ \ \xi\in \kappa_n^{-1}\omega^{\prime}.
\end{align*}
The first term in the above expression is providing the homogenized medium as the second order approximation. 
The second term provides the next order that is the fourth order approximation of the medium 
by considering the last term to be sufficiently small enough. As we know, for each $n$ depending 
upon the parameter $p$, the scales $\varepsilon_{p,n}$ could vary in plenty of ways with remaining inside
the class of sequences of Vitali coverings of $\Omega$. The second order approximation or the homogenized 
tensor $m_{kl}\xi_k\xi_l$ is universal among all possible Vitali coverings, whereas the fourth order
approximation is not so. There is a more vibrant dependence on the scales $\varepsilon_{p,n}$, and it varies 
over the Vitali coverings. 
Taking into account this fact, in order to define the \textit{Burnett coefficient} or the dispersion tensor $d_{HS}$
in the class of generalized Hashin-Shtrikman structures,
we will introduce an approximating quantity $d^{n}_{HS}$ by taking
an average over the various scales $\varepsilon_{p,n}^2$ and then, quotient it out by the highest
scale factor $\kappa_n^2$ $(\kappa_n = \underset{p\in K}{sup}\ \varepsilon_{p,n})$. For that, we will consider the first Bloch eigenvalue 
associated with the shifted operator $\mathcal{A}^{n}_{\omega}(\xi)$ ($\xi\in\mathbb{R}^N$) in $\Omega$. 
Finally, by passing to the limit as $n\rightarrow \infty$, we will characterize the dispersion tensor $d_{HS}$ for the medium.\\
\\
We begin by introducing the following spectral problem in $\Omega$
associated with the shifted operator $\mathcal{A}^{n}(\xi)$ ($\xi\in\mathbb{R}^N$) likewise in \eqref{NE1}: for each fixed $n\in\mathbb{N}$,
\begin{equation}\begin{aligned}\label{dl4}
\mathcal{A}^{n}_\omega(\xi)\varphi^{n}_{\Omega}(x;\xi)=\ -\Big(\frac{\partial}{\partial x_k} + i\xi_k\Big)&\Big[a_{kl}^{n}(x)\Big(\frac{\partial}{\partial x_l} + i\xi_l\Big)\Big]\varphi^{n}_{\Omega}(x;\xi) = \lambda^{n}_{\Omega}(\xi)\varphi^{n}_{\Omega}(x;\xi) \mbox{ in }\Omega, \\
\varphi^{n}_{\Omega}(x;\xi)\mbox{ is constant on }\partial\Omega&\mbox{ and } \int_{\partial\Omega} a_{kl}^{n}(x)\Big(\frac{\partial}{\partial x_l} + i\xi_l\Big)\varphi^{n}_{\Omega}(x;\xi)\nu_k\ d\sigma = 0, 
\end{aligned}\end{equation}
where $\nu$ is the outer normal vector on the boundary and $d\sigma$ is the surface measure on $\partial\Omega$.
\paragraph{Weak formulation:}
Here first we introduce the function spaces
\begin{align*}
L^2_{c}(\Omega) =  \ \{\varphi\in L^2_{loc}(\mathbb{R}^N) \  | \ &\varphi \mbox{ is constant in }\mathbb{R}^N \smallsetminus \Omega \},\\
H^1_{c}(\Omega)=\ \{\varphi\in H^1_{loc}(\mathbb{R}^N) \  | \ &\varphi \mbox{ is constant in }\mathbb{R}^N \smallsetminus \Omega \}.
\end{align*}
Here “$c$” is a floating constant depending on the element under consideration. \\
\\
As a next step we give the weak formulation of the problem in these function spaces. 
We are interested in proving the existence of the eigenvalue and the corresponding eigenvector 
$(\lambda^{n}_{\Omega}(\eta), \varphi^{n}_{\Omega}(x;\xi))\in \mathbb{C}\times H^1_{c}(\Omega)$ of 
the following weak formulation of  \eqref{dl4}: for each fixed $n$,
\begin{multline}\label{dl8}
\int_\Omega a^{n}_{kl}(x)\Big(\frac{\partial \varphi^{n}_{\Omega}(x;\xi)}{\partial x_l} +i\xi_l \varphi^{n}_{\Omega}(x;\xi)  \Big) \overline{ \Big( \frac{\partial \psi}{\partial x_k} + i\xi_k \psi \Big) }dx
\\= \lambda^{n}_{\Omega}(\xi)\int_{\Omega}\varphi^{n}_{\Omega}(x;\xi)\overline{\psi}dx \quad \forall \psi\in H^1_{c}(\Omega).
\end{multline}
\paragraph{Existence Result:}
By following the same analysis presented in \cite{TV1}, we state the corresponding existence result for the problem \eqref{dl8}.
\begin{proposition}
Fix $\xi \in \mathbb{R}^N$. For each fixed $n$, there exist a sequence of eigenvalues $\{ \lambda^{n}_{\Omega,m}(\xi) \geq 0; m \in \mathbb{N} \}$
and its corresponding eigenvectors $\{ \varphi^{n}_{\Omega,m} (x; \xi)\in H^1_{c}(\Omega); m \in \mathbb{N} \}$ satisfying \eqref{dl8}.
\end{proposition}
\paragraph{Regularity of the ground state:}
In the next proposition, we announce the regularity result of ground state based on the Kato-Rellich analysis which has been done in \cite{TV1}.
\begin{proposition}
For each fixed $n\in\mathbb{N}$, we have
\begin{enumerate}
 \item Zero is the first eigenvalue of \eqref{dl8} at $\xi=0$ and it is an isolated point of the spectrum
with its algebraic multiplicity equal to one.
\item There exists an open neighborhood $\Omega^{\prime}_{n}$ around zero such that the 
first eigenvalue $\lambda^{n}_{\Omega,1}(\xi)$ is an analytic function on $\Omega^{\prime}_{n}$ and
there is a choice of the first eigenvector $\varphi^{n}_{\Omega,1}(x;\xi)$ satisfying
$$\xi \longmapsto \varphi^{n}_{\Omega,1}(\cdot;\xi) \in H^1_{c}(\Omega)\mbox{ is analytic on }\Omega^{\prime}_{n}\mbox{ and }\ \varphi^{n}_{\Omega,1}(x;0)= |\Omega|^{-1/2},$$
with the boundary normalization condition 
\begin{equation*}
\varphi_{\omega,1}(y;\eta) =\ \frac{1}{|\omega|^{1/2}} \quad \forall y\in\partial\omega \mbox{ and }\forall\eta\in \omega^{\prime}.
\end{equation*}
\end{enumerate}
\end{proposition}
\paragraph{Derivatives of $\lambda^{n}_{\Omega,1}(\xi)$ and $\varphi^{n}_{\Omega,1}(\xi)$ at $\xi=0$:}
The procedure consists of differentiating the eigenvalue equation \eqref{dl4} 
for $\lambda^{n}_{\omega}(\xi)=\lambda^{n}_{\omega,1}(\xi)$ and $\varphi^{n}_{\Omega}(\cdot;\xi)=\varphi^{n}_{\Omega;1}(\cdot;\xi)$.
\paragraph{Step 1.\hspace{2pt} Zeroth order derivatives:}
We simply recall that $\varphi^{n}_{\Omega,1}(x;0) = |\Omega|^{-1/2}$ by our choice and $\lambda^{n}_{\Omega,1}(0) = 0.$
\paragraph{Step 2.\hspace{2pt} First order derivatives of $\lambda^{n}_{\Omega,1}(\xi)$ at $\xi =0$:}
Differentiating the equation \eqref{dl4} once with respect to $\xi_k$ and then 
taking scalar product with $\varphi^{n}_{\Omega,1}(\cdot;\xi)$ in $L^2(\Omega)$ at $\xi=0$, we get
\begin{equation*}\big\langle D_k (\mathcal{A}^{n}_\omega(0) - \lambda^{n}_{\Omega,1}(0))\varphi^{n}_{\Omega,1}(\cdot;0),\varphi^{n}_{\Omega,1} (\cdot;0)\big\rangle = 0.\end{equation*}
Then, using the fact that
\begin{align*}
D_k \mathcal{A}^{n}_\omega(0)\varphi^{n}_{\Omega,1}(\cdot;0) =\ iC^{n}_k(\varphi^{n}_{\Omega,1}(\cdot;0))&=\ - a^{n}_{kj}(x)\frac{\partial}{\partial x_j}(\varphi^{n}_{\Omega,1}(\cdot;0)) - \frac{\partial}{\partial x_j}(a^{n}_{kj}(y)\varphi^{n}_{\Omega,1}(\cdot;0))\\
                                                                                    &=\  -\frac{\partial}{\partial x_j}(a^{n}_{kj}(x)\varphi^{n}_{\Omega,1}(\cdot;0)),
\end{align*}                                                                                    
whose integral over $\Omega$ vanishes through integration by parts together with using the boundary conditions in \eqref{dl4}, it follows  that
\begin{equation}D_k \lambda^{n}_{\Omega,1}(0) = 0 \quad \forall\hspace{2pt} k = 1,\ldots,N. \end{equation}
\paragraph{Step 3. \hspace{2pt} First order derivatives of $\varphi^{n}_{\Omega,1}(\cdot;\xi)$ at $\xi=0$:}
By differentiating \eqref{dl4} once with respect to $\xi_k$ at zero, one has
\begin{align}
&\mathcal{A}^{n}_\omega(D_k\varphi^{n}_{\Omega,1}(\cdot;0)) =\ -\frac{\partial}{\partial x_j}(a^{n}_{kj}(x)\varphi^{n}_{\Omega,1}(\cdot;0)) \quad\mbox{in }\Omega, \label{dl6} \\
&D_k\varphi^{n}_{\Omega,1}(\cdot;0) =\ 0  \quad\mbox{ on }\ \partial\Omega  \label{dl7}\\
\mbox{and}\quad \int_{\partial\Omega}& A^{n}_\omega(x)\lb \nabla_x  D_k\varphi^{n}_{\Omega,1}(\cdot;0) + i\varphi^{n}_{\Omega,1}(\cdot;0)e_k\rb\cdot\nu \ d\sigma =\ 0 . \label{dl9}
\end{align}
As we can see along with boundary condition \eqref{dl7} for the elliptic equation \eqref{dl6}, 
the solution $D_k\varphi^{n}_{\Omega,1}(\cdot;0)$ gets uniquely determined. And the condition \eqref{dl9} is consistent as it comes 
via integrating the equation \eqref{dl6}. 
By comparing with \eqref{uB}, let us define
\begin{equation}\begin{aligned}\label{dl10}
 D_k\varphi^{n}_{\Omega,1}(x;0)= 
\left\{
 \begin{array}{ll}
 i|\Omega|^{-1/2}\ \varepsilon_{p,n}\big( w_{e_k}(\frac{x-y^{p,n}}{\varepsilon_{p,n}}) - (e_k,\frac{x-y^{p,n}}{\varepsilon_{p,n}})\big)\quad\mbox{ in  }\varepsilon_{p,n}\omega + y^{p,n},\\[1ex]
0    \quad\mbox{ otherwise. }
\end{array}
\right. 
\end{aligned}\end{equation}  
Then, clearly $D_k\varphi^{n}_{\Omega,1}(\cdot;0) \in H^1(\Omega)$ satisfies \eqref{dl7}. We also notice that
\eqref{dl10} solves the equation \eqref{dl6} in each $\{\varepsilon_{p,n}\omega + y^{p,n}\}_{p\in K}$.
In order to show it solves \eqref{dl6} in entire $\Omega$, we need to prove that 
\begin{equation}\label{ll5} \int_{\Omega} A^{n}_\omega(x)\lb \nabla_x D_k\varphi^{n}_{\Omega,1}(x;0) +i\varphi^{n}_{\Omega,1}(x;0) e_k \rb\cdot \nabla_x \varphi(x) dx =\ 0 \quad\forall\varphi\in\mathcal{D}(\Omega).\end{equation}
We have that 
\begin{multline}\label{dl11}
\int_{\Omega} A^{n}_\omega(x)\lb \nabla_x D_k\varphi^{n}_{\Omega,1}(x;0) + i\varphi^{n}_{\Omega,1}(x;0)e_k \rb\cdot \nabla_x \varphi(x) dx\\
=\  \frac{i}{|\Omega|^{1/2}}\sum_p \varepsilon_{p,n}^{N-1}\int_{\omega} A_\omega(y)\nabla_y w_{e_k}(y)\cdot \nabla_y\varphi_p(y) dy,
\end{multline}
where $y=\ \frac{x-y^{p,n}}{\varepsilon_{p,n}} \in \omega$ whenever $ x\in \varepsilon_{p,n}\omega + y^{p,n}$
and $\varphi_p(y) =\ \varphi(\varepsilon_{p,n}y+y^{p,n})$ with $\nabla_y\varphi_p(y)=\ \varepsilon_{p,n}\nabla_x \varphi(x)$.
Then, doing integration by parts on the right hand side of \eqref{dl11} together with using \eqref{hsw}, we get
\begin{align*}
\int_{\Omega} A^{n}_\omega(x)\lb \nabla_x D_k\varphi^{n}_{\Omega,1}(x;0) + i\varphi^{n}_{\Omega,1}(x;0)e_k \rb\cdot \nabla_x \varphi(x) dx &=\ \frac{i}{|\Omega|^{1/2}}\sum_p \varepsilon_{p,n}^{N-1} \int_{\partial\omega} Me_k\cdot \nu \varphi_p(y) d\sigma\\
&=\ \frac{i}{|\Omega|^{1/2}}Me_k \cdot \sum_p \varepsilon_{p,n}^{N-1}\int_{\omega} \nabla_z \varphi_p(y) dy \\
&=\ \frac{i}{|\Omega|^{1/2}}Me_k\cdot \int_{\Omega}\nabla_x \varphi(x) dx\\
&=\ 0 \quad\quad\forall \varphi\in \mathcal{D}(\Omega).
\end{align*}
Thus $D_k\varphi^{n}_{\Omega,1}(x;0)$ is rightly defined in \eqref{dl10} to satisfy \eqref{dl6}, \eqref{dl7}, \eqref{dl9} uniquely.
\paragraph{Step 4.\hspace{2pt} Second derivatives of $\lambda^{n}_{\Omega,1}(\xi)$ at $\xi=0$:}
By differentiating \eqref{dl4} twice with respect to $\xi_k$ and $\xi_l$, respectively
and then taking scalar product with $\varphi^{n}_{\Omega,1}(\cdot;\xi)$ in $L^2(\Omega)$ at $\xi=0$, we get
\begin{multline*}
\big\langle D^2_{kl}(\mathcal{A}^{n}_\omega(0) - \lambda^{n}_{\Omega,1}(0))\varphi^{n}_{\Omega,1}(\cdot;0), \varphi^{n}_{\Omega,1}(\cdot;0)\big\rangle \\
+ \big\langle [D_k(\mathcal{A}^{n}_\omega(0) - \lambda^{n}_{\Omega,1}(0))]D_l\varphi^{n}_{\Omega,1}(\cdot;0), \varphi^{n}_{\Omega,1}(\cdot;0)\big\rangle\\
+ \big\langle [D_l(\mathcal{A}^{n}_\omega(0)- \lambda^{n}_{\Omega,1}(0))]D_k \varphi^{n}_{\Omega,1}(\cdot;0), \varphi^{n}_{\Omega,1}(\cdot;0)\big\rangle = 0.
\end{multline*}
By using the information obtained in the previous steps, we get
\begin{equation*}\begin{aligned}
\frac{1}{2}&D^2_{kl}\lambda^{n}_{\Omega,1}(0) =\ \frac{1}{|\Omega|}\int_{\Omega} a^{n}_{kl}(x)dx -
\frac{1}{2|\Omega|}\int_{\Omega}\Big[C^{n}_k(D_l\varphi^{n}_{\Omega,1}(x;0)) + C_l(D_k\varphi^{n}_{\Omega,1}(x;0))\Big]dx\\
&=\ \frac{1}{2|\Omega|}\sum_p \int_{\varepsilon_{p,n}\omega +y^{p,n}}A_\omega\Big(\frac{x-y^{p,n}}{\varepsilon_{p,n}}\Big)\Big( \nabla w_{e_k}\Big(\frac{x-y^{p,n}}{\varepsilon_{p,n}}\Big)\cdot e_l + \nabla w_{e_l}\Big(\frac{x-y^{p,n}}{\varepsilon_{p,n}}\Big)\cdot e_k \Big) dx\\
&=\  \frac{1}{2} \int_{\omega} A_\omega(y)\big( \nabla_y w_{e_k}(y)\cdot e_l\ + \nabla_y w_{e_l}(y)\cdot e_k\big) dy\\
&=\ m_{kl} \quad\forall  k,l=1,\ldots,N,
\end{aligned}
\end{equation*}
due to the integral identity \eqref{dl12},
which are indeed the homogenized coefficients governed with the Hashin-Shtrikman constructions.
We see that $\frac{1}{2}D^2_{kl}\lambda^{n}_{\Omega,1}(0)$ is independent of $n$. Thus, it
does not depend on the choice of translations $y^{p,n}$ and the scales $\varepsilon_{p,n}$
as long as they are bound to satisfy the Vitali covering criteria \eqref{ll1}.
\paragraph{Step 5. Higher order derivatives:}
In general, the process can be continued indefinitely to compute all derivatives of
$\lambda^{n}_{\Omega,1}(\xi)$ and $\varphi^{n}_{\Omega,1}(\cdot;\xi)$ at $\xi = 0$. 
In particular, the third order derivative is zero, i.e. $D^{q}\lambda^{n}_{\Omega,1}(0)=0, \,|q|=3$.
However, we are interested in the fourth order derivatives of $\lambda^{n}_{\Omega,1}(0)$, i.e.
$D^4_{klmn}\lambda^{n}_{\Omega,1}(0)$, which is in general a non-positive definite tensor and can be defined as follows: the second order derivative of the eigenvector $D^2_{kl}\varphi^{n}_{\Omega,1}(\cdot;0)\in H^1_0(\Omega)$ solves
\begin{equation}\begin{aligned}\label{dl15}
&\mathcal{A}^{n}_\omega D^2_{kl}\varphi^{n}_{\Omega,1}(x;0)= -( a^{n}_{kl}(x) - m_{kl} )\varphi^{n}_{\Omega,1}(x;0) - iC^{n}_k(D_l(\varphi^{n}_{\Omega,1}(x;0)) - iC_l(D_k\varphi^{n}_{\Omega,1}(x;0))\mbox{ in }\Omega,\\
&D^2_{kl}\varphi^{n}_{\Omega,1}(x;0) = 0 \mbox{ on }\partial\Omega \mbox{ and } \int_{\partial\Omega} A^{n}_\omega(x)\nabla_x D^2_{kl}\varphi^{n}_{\omega,1}(x;0)\cdot\nu \ d\sigma= 0 .
\end{aligned}\end{equation}
The above equation \eqref{dl15} has an unique solution and we would like to define as 
\begin{equation}\begin{aligned}\label{dl16}
D^2_{kl}\varphi^{n}_{\Omega,1}(x;0) =
\left\{
\begin{array}{ll}
 |\Omega|^{-1/2}\ \varepsilon^2_{p,n} \widetilde{w}_{kl}\big(\frac{x-y^{p,n}}{\varepsilon_{p,n}}\big)\quad\mbox{ in  }\varepsilon_{p,n}\omega + y^{p,n},\\[1ex]
0    \quad\mbox{ otherwise,}
\end{array}
\right. 
\end{aligned}\end{equation}  
where $\widetilde{w}_{kl}$ is defined likewise by \eqref{hsw} as follows: after extending $A_\omega\in\mathcal{M}(\alpha,\beta;\omega)$ by $A_\omega(x) = M$ for $x \in \mathbb{R}^N\smallsetminus \omega$, $\widetilde{w}_{kl}\in H^1(\mathbb{R}^N)$  satisfies
\begin{equation}\begin{aligned}\label{dl14}
- div (A_\omega\nabla \widetilde{w}_{kl}(x)) &= -( a^\omega_{kl}(x) - m_{kl} ) - iC^{\omega}_k(w_{e_l}(x)-x_l) - iC^{\omega}_l(w_{e_k}(x)-x_k)\quad\mbox{ in }\mathbb{R}^N,\\
  \widetilde{w}_{kl}(x) &=\ 0 \quad\mbox{ in }\mathbb{R}^N \smallsetminus \omega,
\end{aligned}\end{equation}
where $C^{\omega}_k(\varphi) = - a^{\omega}_{kj}(x)\frac{\partial \varphi}{\partial x_j} - \frac{\partial}{\partial x_j}(a^{\omega}_{kj}(x)\varphi)$.\\
\\
If such $\widetilde{w}_{kl}\in H^1(\mathbb{R}^N)$ exists $\forall k,l=1,\ldots,N$, then following the same arguments
presented in Step 3, $D^2_{kl}\varphi^{n}_{\Omega,1}(\cdot;0)$ defined in \eqref{dl16} belongs to $ H^1_0(\Omega)$ and solves \eqref{dl15} in $\Omega$.\\
\\
Notice that $w_{kl}\in H^1_0(\omega)$ defined in \eqref{dl17} and $\widetilde{w}_{kl}\in H^1_0(\omega)$ solves the same equation \eqref{dl14}.
The only difference occurs in the co-normal derivative of $\widetilde{w}$ on $\partial\omega$, because we have
\begin{center}$\nabla \widetilde{w}_{kl}(x)\cdot \nu =\ 0 \ \mbox{ on }\partial\omega.$\end{center}
In the next section (cf. Proposition \ref{ll2} below), we show that in the case of two-phase spherical 
inclusions (see Example \ref{si}) such $\widetilde{w}_{kl}$ exists and is equal to $w_{kl}$ for each $k,l=1,\ldots,N.$ \\
\\
We further define
\begin{center}$ X^{n}_{\Omega,1}= -i|\Omega|^{1/2}\ \xi_kD_k\varphi^{n}_{\Omega,1}(\cdot;0)$ and  $X^{n}_{\Omega,2}= |\Omega|^{1/2}\ \xi_k\xi_l D^2_{kl}\varphi^{n}_{\Omega,1}(\cdot;0).$\end{center}
Then, following \cite[Proposition 3.2]{COVB}, the fourth order derivative of $\lambda^{n}_{\Omega,1}(\xi)$ at $\xi=0$ defines as 
\begin{equation}\begin{aligned}\label{dl18}
\frac{1}{4!}D^4_{klmn}\lambda^{n}_{\Omega,1}(0)\xi_k\xi_l\xi_m\xi_n&= -\frac{1}{|\Omega|}\int_{\Omega} \mathcal{A}^{n}\Big( X^{n}_{\Omega,2} - \frac{1}{2}(X^{n}_{\Omega,1})^2\Big)\cdot\Big( X^{n}_{\Omega,2} - \frac{1}{2}(X^{n}_{\Omega,1})^2\Big)  dx \\
&\leq 0.
\end{aligned}\end{equation}
Moreover, using \eqref{dl10} and \eqref{dl16}, the above expression \eqref{dl18} becomes 
\begin{equation}\begin{aligned}\label{ll3}
 \frac{1}{4!}D^4_{klmn}\lambda^{n}_{\Omega,1}(0)\xi^4 &=\ -\frac{1}{|\Omega|}\sum_p \varepsilon_{p,n}^{N+2} \int_{\omega} \mathcal{A}\Big( X^{(2)}_{\omega} - \frac{1}{2}(X^{(1)}_{\omega})^2\Big)\cdot\Big( X^{(2)}_{\omega} - \frac{1}{2}(X^{(1)}_{\omega})^2\Big)  dy \\
                                                      &=\ \frac{|\omega|}{|\Omega|}\sum_p \varepsilon_{p,n}^{N+2}\cdot \frac{1}{4!}D^4_{klmn}\lambda_{\omega,1}(0)\xi^4,
\end{aligned}\end{equation}
where the last equality follows from \eqref{dl20}.
\begin{remark}
The above equality \eqref{ll3} establishes the relation between the fourth order derivatives of $\lambda_{\omega,1}$ and $\lambda_{\Omega,1}$.
Remember that the first and second order derivatives of them are equal.
\end{remark}
Here we define an approximating dispersion tensor $d^{n}_{HS}$ for the medium with respect to the highest scale factor $\kappa_n^2$ as follows:
$$ \frac{1}{4!}D^4_{klmn}\lambda^{n}_{\Omega,1}(0) =\ \kappa_n^2 d^{n}_{HS}.$$
Then, as $n\rightarrow \infty$, we define the \textit{Burnett coefficient} or the dispersion tensor for the medium:
\begin{equation}\label{dl21} d_{HS}=\ \underset{n \rightarrow \infty}{limsup\ }d^{n}_{HS} =\  \frac{|\omega|}{|\Omega|}\lb\underset{n\rightarrow \infty}{limsup\ }\kappa_n^{-2}\sum_p \varepsilon_{p,n}^{N+2}\rb\cdot \frac{1}{4!}D^4_{klmn}\lambda_{\omega,1}(0).\end{equation}
The above limit always exists finitely. It can be seen through the following simple estimate:\begin{align*}
\sum_p\varepsilon_{p,n}^{N+2} \leq \kappa_n^2 \sum_p \varepsilon_{p,n}^N = \kappa_n^2 \frac{|\Omega|}{|\omega|}\quad\mbox{ or } \quad\kappa_n^{-2}\sum_p \varepsilon_{p,n}^{N+2} \ \mbox{ is uniformly bounded. }
\end{align*}
The identity  \eqref{dl21} reads as $d_{HS}$ is a purely locally defined macro quantity incorporating only various scales associated with the structure. For each $n$, we have the following approximation: 
$$ \lambda^{n}_{\Omega,1}(\xi)=\ M\xi^2 + \kappa_n^2\ d_{HS}\ \xi^4 + o(\kappa_n^2), \quad \xi\in\Omega^{\prime}_{n}.$$
\begin{remark}
Remember that the above expression \eqref{dl21} is valid only when $A$ is equivalent to $M$ through the existence of 
$w_{e_k}\in H^1_{loc}(\mathbb{R}^N)$ satisfying \eqref{hsw} and with the existence of $\widetilde{w}_{kl}\in H^1(\mathbb{R}^N)$ satisfying \eqref{dl14}, for each $k,l=1,\ldots,N$.
In the next section, as an example of ``Spherical inclusions in two-phase medium'', we establish their existence. 
\end{remark}
\begin{remark}
For the periodic micro-structures with the uniform $\varepsilon$-scaling 
and translation, the above definition \eqref{dl21} of the dispersion tensor coincides with the coefficient  $d_{Y}$ defined in
\eqref{dl22}.\hfill\qed
\end{remark}

Motivated from the optimal design and so on, 
the interesting question can be taken into account in 
this matter that which Vitali coverings are responsible 
for the minimum or maximum value for $d_{HS}$. Regarding that, we prove 
the following conjecture stated below.
\paragraph{Conjecture:}\textit{Minimizer of the dispersion tensor is unique among $2$-phase periodic Hashin-Shtrikman micro-structures of a given proportion and it is given by the
Apollonian-Hashin-Shtrikman micro-structure.}\\
\\
This conjecture was arrived at by a previous study of the same problem in
one-space dimension \cite{CMSV}. Roughly speaking, the result in one dimension says 
that the value of ``$d$'' increases when we increase the number of interfaces between
the two phases in the micro-structure. At the maximum value of ``$d$'', we have a
continuum of interfaces and at the minimum value, there is an unique
minimizer with a single interface. We prove it in the following section. 

\section{Spherical inclusions in 2-phase periodic Hashin- Shtrikman micro-structures}
\setcounter{equation}{0}

In the class of periodic spherical Hashin-Shtrikman micro-structures,
we consider the unit cell $Y=[0,1]^N$ in $\mathbb{R}^N$ and identify with $\mathbb{R}^N$ 
through $\mathbb{Z}^N-$translation invariance. We first  find a 
Hashin-Shtrikman construction to cover the whole space $\mathbb{R}^N$ and if it is invariant 
under $\mathbb{Z}^N-$translations, then we will consider it as a Hashin-Shtrikman structure
for $Y$ and conversely. So, let us start with a cover for $\mathbb{R}^N$ by a sequence 
of reduced copy of disjoint balls $B(y^p,\varepsilon_p)= \varepsilon_{p}B(0,1) + y^{p}$ with center $y^{p}$ and  radius $\varepsilon_p$ such that
\begin{equation}\label{ll7}
\begin{aligned}
 & meas\big(\mathbb{R}^N \smallsetminus \underset{p\in \mathbb{K}}{\cup} B(y^p,\varepsilon_p)\big) = 0,\mbox{ where }K \mbox{ is some infinite countable set}\\
\mbox{and }\ & m\in\mathbb{Z}^N,\ \forall p\in {K}, \ m+ B(y^p,\varepsilon_p) = B(y^p+m,\varepsilon_p) \in \underset{p\in K}{\cup} B(y^p,\varepsilon_p).\\
\mbox{Moreover, }\ &\forall m\in \mathbb{Z}^N, \ meas(\mathbb{R}^N \smallsetminus \underset{p\in \mathbb{K}}{\cup} (m+ B(y^p,\varepsilon_p)) ) = 0.
\end{aligned}
\end{equation}
Consequently, the unit periodic cell $Y$ is understood as $Y= \underset{p\in K}{\cup}\big( [0,1]^N\cap B (y^p,\varepsilon_p)\big)$.\\
\\
Let us now consider $a_B(y)$ be the two-phase conductivity profile in  $B(0,1)$, defined as
follows:\begin{equation*}\begin{aligned}
a_B(y)=\ a(r) =
\left\{
\begin{array}{ll} 
\alpha  \quad\mbox{if } |y| < R,\\
\beta  \quad\mbox{if }  R < |y| < 1,
\end{array}
\right. 
\end{aligned}\end{equation*}
with  $0<\alpha\leq \beta < \infty$. We define $\theta = R^N$ as the volume proportion of the two-phase profile.\\
\\
Then $a_B(y)$ is \textit{equivalent} to some $m$ ($\alpha \leq m\leq \beta$) (see \cite[Page no. 282]{T}), i.e., after extending $a_B(y)$ by $a_B(y)=m$ in $\mathbb{R}^N\smallsetminus B(0,1)$, 
for each unit vector $e_l \in \mathbb{R}^N$ ($l=1,\ldots,N$), there exists $w_{e_l}\in  H^{1}_{loc}(\mathbb{R}^N)$ satisfying\begin{equation}\label{ED1}
- div(a_B(y)\nabla w_{e_l}(y))= 0 \quad\mbox{in }\mathbb{R}^N,\quad  w_{e_l}(y) = y\cdot e_l  \quad\mbox{in }\mathbb{R}^N \smallsetminus B(0,1).
\end{equation}
The co-normal flux satisfies 
\begin{equation*}a_B(y)\nabla w_{e_l}(y)\cdot \nu =m \quad\mbox{on }\partial B(0,1),\end{equation*}
where $m$ satisfies the relation
\begin{equation}\label{ED9}\frac{m - \beta}{m + (N-1)\beta} = \theta \frac{\alpha - \beta}{\alpha + (N-1)\beta}.\end{equation} 
Now, by homothecy, we extend  $a_B$ to the entire $\mathbb{R}^N$ defining
$$a_{\mathbb{R}^N}(y)= a_B\Big(\frac{y-y^p}{\varepsilon_p}\Big)\quad\mbox{in }B(y^p,\varepsilon_p)\ \mbox{ a.e. in }\mathbb{R}^N$$
and reveal that $a_{\mathbb{R}^N}(y)$ is a $Y-$periodic function due to \eqref{ll7}, i.e. $a_{\mathbb{R}^N}\in L^{\infty}_{\#}(Y)$. 
We define 
\begin{center}$a_Y(y) =\ a_{\mathbb{R}^N}(y),\ \ $ $y\in Y=[0,1]^N$.\end{center}
Next, we set
\begin{center}$ a^{\varepsilon}(x) = a_{Y}\big(\frac{x}{\varepsilon}\big),\ x\in\mathbb{R}^N\ \mbox{and } \frac{x}{\varepsilon}= y\in Y$\end{center}
and extend it to the whole $\mathbb{R}^N$ by $\varepsilon$-periodicity with a small period of scale $\varepsilon$, 
which is considered as two-phase periodic Hashin-Shtrikman micro-structures with spherical inclusions.
\paragraph{(i) Homogenized coefficients:}
\noindent
The sequence $a^{\varepsilon} \xrightarrow{ H-\mbox{converges }} m$. 
One defines $\chi_{l}\in H^1_{\#}(Y)$ (see \cite[page no. 195]{JKO}) solving the cell--problem in the periodic cell $Y$:
\begin{equation}\label{PN1} 
-div(a_{Y}(y)(\nabla{\chi_l}(y) + e_l)) = 0 \ \mbox{ in }Y,\ \mbox{ where } \chi_l\in H^1_{\#}(Y)\ \mbox{ with }\ \int_Y \chi_l dy = 0.
\end{equation} 
And then, one need to show that the homogenized coefficient $a^{*}$, defined below, is equal to $m$, i.e. 
\begin{equation}\label{ED3}
 a^{*}= \frac{1}{|Y|}\int_{Y} a_{Y}(y)(\nabla\chi_l(y) + e_l)\cdot(\nabla\chi_l(y) + e_l)dy\ =\ m.
\end{equation}
Let us first look for a solution of the following extended equation in the entire space $\mathbb{R}^N$: 
\begin{equation}\label{ll4}
-div (a_{\mathbb{R}^N}(y)(\nabla \chi_{\mathbb{R}^N} + e )) = 0 \ \mbox{ in }\mathbb{R}^N,\ \chi_{\mathbb{R}^N}\in H^1_{loc}(\mathbb{R}^N),
\end{equation}
where $e$ is some canonical basis vector in $\mathbb{R}^N$.\\
\\
Prior to that, we define 
\begin{equation}\label{ll6}
\chi_{\mathbb{R}^N}(y) =\ \varepsilon_{p}\Big( w_{e}\Big(\frac{y-y^{p}}{\varepsilon_{p}}\Big) - \Big(e\cdot\frac{y-y^{p}}{\varepsilon_{p}}\Big) \Big)\ \ \mbox{ if  }y\in B(y^p,\varepsilon_p).
\end{equation}
Then, we see that $\chi_{\mathbb{R}^N}$ is a $H^1_{loc}(\mathbb{R}^N)-$function and  it solves the problem  \eqref{ll4} restricted into each balls $\{B(y^p,\varepsilon_p)\}_{p\in K}$.
Moreover, for any $\varphi \in \mathcal{D}(\mathbb{R}^N)$, we have 
$$ \int_{\mathbb{R}^N} a_{\mathbb{R}^N}(y)(\nabla\chi_{\mathbb{R}^N}(y) + e)\cdot \nabla \varphi (y)\ dy =0.$$
The above equality follows in a similar manner that we did before for \eqref{ll5}. It establishes that \eqref{ll6} solves \eqref{ll4} locally in $\mathbb{R}^N$.\\
Now, we claim that, $\chi_{\mathbb{R}^N}(y)$ is a $Y-$periodic function, i.e. $\chi_{\mathbb{R}^N}\in H^1_{\#}(Y)$.
It simply follows by using \eqref{ll7}, i.e. for $y\in \mathbb{R}^N$ and $m\in \mathbb{Z}^N$, we have
\begin{align*}
\chi_{\mathbb{R}^N}(y-m) &=\ \varepsilon_p\Big( w_{e}\Big(\frac{y-m-y^{p}}{\varepsilon_{p}}\Big) - \Big(e\cdot\frac{y-m-y^{p}}{\varepsilon_{p}}\Big) \Big)\ \mbox{ if } y-m \in B(y^p,\varepsilon_p)\\
                         &=\  \varepsilon_{p}\Big( w_{e}\Big(\frac{y-y^{p^{\prime}}}{\varepsilon_{p}}\Big) - \Big(e\cdot\frac{y-y^{p^{\prime}}}{\varepsilon_{p}}\Big) \Big)\ \mbox{ if } y \in B(y^{p^{\prime}},\varepsilon_{p})= B(m+y^p,\varepsilon_p)\\ 
                         &=\ \chi_{\mathbb{R}^N}(y) \quad\mbox{ (due to }\eqref{ll7} ).
\end{align*}
We define
\begin{center}$ \widetilde{\chi_{Y}}(y)=  \chi_{\mathbb{R}^N}(y) ,\ \ $ $y\in Y=[0,1]^N$.\end{center}
Then $\widetilde{\chi_{Y}}(y)\in H^1_{\#}(Y)$ and by simply considering $\chi_{Y}(y)= \widetilde{\chi_{Y}}(y)- \frac{1}{|Y|}\int_Y\widetilde{\chi_{Y}}(y)dy$, it
solves \eqref{PN1} for each $e= e_l$. Finally, by taking $\chi_Y(y)$ in the integral identity \eqref{ED3} and using \eqref{ED1}, it follows the homogenized coefficient $m$. Precisely, we have 
\begin{align*}
\frac{1}{|Y|}\int_{Y} a_Y(y)(\nabla \chi_Y(y) + e)&\cdot(\nabla \chi_Y(y) + e) dy  \\
=& \ \frac{1}{|Y|}\sum_p \int_{B(y^p,\varepsilon_p)\cap Y} a\Big(\frac{y-y^{p}}{\varepsilon_{p}}\Big)\Big|\nabla w_{e}\Big(\frac{y-y^{p}}{\varepsilon_{p}}\Big)\Big|^2 dy\\
=&\ \frac{1}{|B(0,1)|} \int_{B(0,1)} a(z) |\nabla w_{e_l}(z)|^2 dz\ =\ m.
\end{align*}
It is now remaining to establish the relation \eqref{ED9}. We seek the solution of the above equation \eqref{ED1} in the following form 
\begin{equation}\label{ED7}w_{e_l}(y) =\ y_lf(r),\ y\in B(0,1),\end{equation}
where $f(r)$ is given by
\begin{equation}\begin{aligned}\label{ED5}
f(r) =
\left\{
\begin{array}{ll}
 \widetilde{b_1} \ &\mbox{ if }r< R,\\[1ex]
\widetilde{b_2} + \frac{\widetilde{c}}{r^N} \ &\mbox{ if }R < r < 1,\\[1ex]
1 \ &\mbox{ if }1 < r.
\end{array}
\right.
\end{aligned}\end{equation}
In order to keep the solution $w_{e_l}(y)$ and flux $a(r)(f(r)+rf^{\prime}(r))$ to be continuous across
the inner boundary $(r=R)$ and the outer boundary $(r=1)$,
we need to impose the following conditions: \begin{equation}\begin{aligned}\label{b1b2}
\widetilde{b_1} =\ \widetilde{b_2} + \frac{\widetilde{c}}{r_1^N}, &\ \mbox{  }\ \alpha\widetilde{b_1} =\ \beta\Big(\widetilde{b_2} + \frac{(1-N)\widetilde{c}}{r_1^N} \Big),\\
\widetilde{b_2} + \widetilde{c} =\ 1 &\ \mbox{ and }\ \beta(\widetilde{b_2} + (1-N)\widetilde{c}) =\ m. \\
\end{aligned}\end{equation}
Then, solving $(\widetilde{b_1},\widetilde{b_2},\widetilde{c})$ in terms of $(\alpha,\ \beta,\ \theta)$ from the first three equation of \eqref{b1b2},  we have
\begin{equation}\label{PN2}
\widetilde{b_1} = \ \frac{N\beta}{(1-\theta)\alpha + (N+\theta -1)\beta},\quad \widetilde{b_2} = \ \frac{(1-\widetilde{b_1}\theta)}{(1-\theta)} \quad\mbox{and}\quad \widetilde{c} = \frac{(\widetilde{b_1} -1)\theta}{(1-\theta)}
\end{equation}
and finally putting it into the fourth equation of \eqref{b1b2}, $m$ can be written as in \eqref{ED9}.

\paragraph{(ii) Dispersion coefficient:}
In the periodic Hashin-Shtrikman structures we denote the dispersion tensor by $d_{PHS}$. 
Concerning to our case, we recall the integral expression \eqref{dy} to write $d_{PHS}$ as follows:
\begin{equation}\label{dphs}
d_{PHS}\eta^4 = -\frac{1}{|Y|}\int_{Y} \mathcal{A}_Y\Big( X^{(2)}_Y -\frac{(X^{(1)}_Y)^2}{2}\Big)\cdot\Big( X^{(2)}_Y -\frac{(X^{(1)}_Y)^2}{2}\Big) dy,
\end{equation}
where $\mathcal{A}_Y$, $X^{(1)}_Y$, $X^{(2)}_Y$ are defined in \eqref{dl19}.\\
\\
Let us denote $X_{B(0,1)}^{(1)}(y) = \eta_k ( w_{e_k}(y) - y_k )$, where $w_{e_k}$ is the solution of \eqref{ED1}
and  $X_{B(0,1)}^{(2)}= \eta_k\eta_l w_{kl}$, where $w_{kl}$ is the solution of  the following auxiliary cell-equation in $B(0,1)$:
\begin{equation}\begin{aligned}\label{auxN}
-div(a_B(y)\nabla w_{kl}(y)) &=\ a_B(y)\delta_{kl} - m\delta_{kl} - \frac{1}{2}\lb C^B_{l}(w_{e_k}(y)-y_k) + C^B_{k}(w_{e_l}(y) -y_l)\rb  \mbox{ in }B(0,1),\\
w_{kl}(y) &=\ 0 \ \mbox{ on }\partial B(0,1),
\end{aligned}\end{equation} 
where $C^{B}_k(\varphi) = - a_B(y)\frac{\partial \varphi}{\partial x_k} - \frac{\partial}{\partial x_k}(a_B(y)\varphi)$.\\
\\
We observe that \eqref{auxN} is an elliptic partial differential equation with  
Dirichlet boundary condition which possess an unique solution $w_{kl} \in H^1_0(B(0,1))$. 
Having that, we claim that the co-normal derivative of $w_{kl}$ on $\partial B(0,1)$ is 
zero, i.e.
\begin{equation}\label{fed12}
\nabla w_{kl} (y)\cdot \nu = 0 \quad\mbox{on } \partial B(0,1).
\end{equation}
\begin{remark}
If we extend $a_{B}(y)$ by $m$ in $\mathbb{R}^N\smallsetminus B(0,1)$ and $w_{e_k}(y)$ by $y_k$ in $\mathbb{R}^N\smallsetminus B(0,1)$, then \eqref{auxN} becomes 
\begin{equation}\label{ED6}-div(m\nabla \widetilde{w_{kl}}(y)) = 0 \quad\mbox{in }\ \mathbb{R}^N\smallsetminus B(0,1)\quad\mbox{with }\ \ \widetilde{w_{kl}}(y) = 0 \ \mbox{on }\partial B(0,1).\end{equation}
If $ \widetilde{w_{kl}} \in H^1(\mathbb{R}^N\smallsetminus B(0,1))$, then simply using the maximum principle (see \cite[Page no. 164, (3.10)]{KS}), we
get $\widetilde{w_{kl}}(y) = 0 $ in $\mathbb{R}^N\smallsetminus B(0,1)$, which says that $0$ is the natural extension. \\
\\
Let us define
\begin{equation}\begin{aligned}\label{ll8}
 \widetilde{w_{kl}} =
 \left\{
 \begin{array}{ll}
  w_{kl} \ \mbox{ in } B(0,1),\\[1ex]
  0 \ \mbox{ in }\mathbb{R}^N\smallsetminus B(0,1).
  \end{array}
  \right.
\end{aligned}\end{equation}
Now, if \eqref{ll8} solves both \eqref{auxN} and \eqref{ED6} as a $H^1(\mathbb{R}^N)-$function,
then from the continuity of the boundary normal flux, we have
$$ a_B(y)\nabla w_{kl} (y)\cdot \nu = 0 \quad\mbox{on } \partial B(0,1)\quad\mbox{or }\ \nabla w_{kl} (y)\cdot \nu = 0 \quad\mbox{on } \partial B(0,1).$$
However, at this moment we don't know whether $\widetilde{w_{kl}}$ is a $H^1(\mathbb{R}^N)-$function or not.\\
Secondly, as we have experienced from the previous case, it is required to have such extension property in order to get 
$\chi_{kl}\in H^1_{\#}(Y)$ from $w_{kl}\in H^1_0(B)$, which solves the cell-problem \eqref{ll9}.
\end{remark}
\begin{proposition}\label{ll2}
 The unique solution $w_{kl}$ of \eqref{auxN} satisfies  the additional boundary condition \eqref{fed12}.
\end{proposition}
\begin{proof}
The proof is divided into several steps. 
We begin with calculating the right hand side of the equation \eqref{auxN}.
\paragraph{Step 1)} \textbf{RHS of \eqref{auxN}: }
Following the definition of the 1st order operator $C^B_l$ and $w_{e_k}(y) = y_kf(r)$, we get
\begin{align*}
- C^B_l(w_{e_k}(y)-y_k) &=\ a(r)\frac{\partial}{\partial y_l} (w_{e_k}(y)-y_k) + \frac{\partial}{\partial y_l}(a(r)(w_{e_k}(y)-y_k))\\
                      %&=\ a(r)\frac{\partial}{\partial y_l} (y_k(h(r)-1)) + \frac{\partial}{\partial y_l}(a(r)(y_k(h(r)-1)))\\
                      &=\ 2a(r)(f(r)-1)\delta_{kl} + y_ky_l \Big( a(r)\frac{f^{\prime}(r)}{r} + \frac{(a(r)(f(r)-1))^{\prime}}{r} \Big).
                      %&=\ 2a(r)\{ \delta_{kl}(f(r)-1) + y_ky_l \frac{f^{\prime}(r)}{r}\} + \{ a(r)(f(r)-1)\delta_{kl} + y_ky_l\frac{(a(r)(f(r)-1))^{\prime}}{r}\}
\end{align*}
Or,
\begin{multline}\label{fed2}
a(r)\delta_{kl} - m\delta_{kl} -\frac{1}{2}\big( C^B_l(w_{e_k}(y)-y_k) +  C^B_k(w_{e_l}(y)-y_l) \big) \\
= a(r)\delta_{kl} - m\delta_{kl} + 2a(r)(f(r)-1)\delta_{kl} + y_ky_l \Big( a(r)\frac{f^{\prime}(r)}{r} + \frac{(a(r)(f(r)-1))^{\prime}}{r} \Big).
\end{multline}
The structure of RHS suggests the following ansatz for the solution of \eqref{auxN}:
\begin{equation*}
 w_{kl}(y) = y_ky_l g(r) + h(r). 
\end{equation*}
\textbf{LHS of \eqref{auxN}:} We have
\begin{align*}
 \frac{\partial w_{kl}}{\partial y_m}(y) = y_ky_l g^{\prime}(r)\frac{y_m}{r} + y_kg(r)\delta_{lm} + y_lg(r)\delta_{km} + h^{\prime}(r)\frac{y_m}{r}. 
\end{align*}
Consequently,
\begin{align*}
\frac{\partial}{\partial y_m} \Big( a(r)\frac{\partial w_{kl}}{\partial y_m}(y)\Big) =&\ y_ky_ly_m \Big(\frac{a(r)g^{\prime}(r)}{r}\Big)^{\prime}\ \frac{y_m}{r} + y_ky_l \frac{a(r)g^{\prime}(r)}{r}
+ y_k\delta_{lm}y_m\frac{a(r)g^{\prime}(r)}{r}\\
& +y_l\delta_{km}y_m\frac{a(r)g^{\prime}(r)}{r} + 
2\delta_{km}\delta_{lm}a(r)g(r) 
+ y_ky_m\delta_{lm}\frac{(a(r)g(r))^{\prime}}{r}\\
& + y_ly_m\delta_{km}\frac{(a(r)g(r))^{\prime}}{r}
+ \frac{a(r)h^{\prime}(r)}{r} + y_m\Big(\frac{a(r)h^{\prime}(r)}{r}\Big)^{\prime}\ \frac{y_m}{r},
\end{align*}
or\begin{equation}\label{fed3}\begin{aligned}
div (a(r)\nabla w_{kl}(y)) =&\ y_ky_l \Big( r \Big(\frac{a(r)g^{\prime}(r)}{r}\Big)^{\prime} + (N+2)\frac{a(r)g^{\prime}(r)}{r} + 2\frac{(a(r)g(r))^{\prime}}{r} \Big)\\ &+ 2a(r)g(r)\delta_{kl}
                            + N \frac{a(r)h^{\prime}(r)}{r} + r\Big(\frac{a(r)h^{\prime}(r)}{r}\Big)^{\prime}. 
\end{aligned}\end{equation}
\textbf{Step 2)}
Both LHS \eqref{fed3} and RHS \eqref{fed2} contain the quadratic term $y_ky_l$ and the constant term in $y$. Equating the corresponding coefficients, we get
\begin{equation}\label{fed4}
r \Big(\frac{a(r)g^{\prime}(r)}{r}\Big)^{\prime} + (N+2)\frac{a(r)f^{\prime}(r)}{r} + 2 \frac{(a(r)g(r))^{\prime}}{r} = - \Big[ a(r)\frac{f^{\prime}(r)}{r} + \frac{(a(r)(f(r)-1))^{\prime}}{r} \Big]
\end{equation}
and 
\begin{equation}\label{fed5}
 2a(r)g(r)\delta_{kl} + N \frac{a(r)h^{\prime}(r)}{r} + r\Big(\frac{a(r)h^{\prime}(r)}{r}\Big)^{\prime} = - \Big[  a(r)\delta_{kl} - m\delta_{kl} + 2a(r)(f(r)-1)\delta_{kl}\Big].
\end{equation}
We have
\begin{equation*}
 a(r) =
  \alpha \quad\mbox{and }\ \ f(r) =\ \widetilde{b_1} \quad\mbox{when } r < R \quad
\end{equation*}
and
\begin{equation*}
  a(r)  =\beta  \quad\mbox{and }\  \  f(r)=\ \widetilde{b_2} + \frac{\widetilde{c}}{r^N} \quad\mbox{when } R< r <1,
\end{equation*}
where ($\widetilde{b_1},\widetilde{b_2},\widetilde{c}$) are known in terms of $\alpha,\beta, N $ and $ \theta$ (see \eqref{PN2}).\\
\\
We further seek $h(r)$ and $g(r)$ in the general form of
\begin{align}\label{fed6}
g(r) =\ b + \frac{c}{r^N} + \frac{d}{r^{N+2}} \quad\mbox{and }\ h(r) =\Big (\frac{p}{r^N} + q r^2 + t\Big)\delta_{kl}.  
\end{align}
The set of constants $(b,c,d)$ and $(p,q,t)$ can take different values in 
the ranges $r < R $ and $ R< r<1$. We denote them
by $(b_1,c_1,d_1)$, $(p_1,q_1,t_1)$ and $(b_2,c_2,d_2)$, $(p_2,q_2,t_2),$ respectively.\\ 
Now, by using \eqref{fed6} in \eqref{fed4}, we get the following cases:\\

{\it Case 1.} When $r < R$, we have
\begin{multline*}
-N(N+2)\frac{c_1}{r^{N+2}} - (N+2)^2\frac{d_1}{r^{N+4}} + N(N+2)\frac{c_1}{r^{N+2}} \\+  (N+4)(N+2)\frac{d_1}{r^{N+4}}+ 2 \Big(- N\frac{c_1}{r^{N+2}} - (N+2)\frac{d_1}{r^{N+4}} \Big) = 0,
\end{multline*}
or
\begin{multline*}
\frac{c_1}{r^{N+2}} ( -N(N+2) + N(N+2) -2N ) 
\\
+ \frac{d_1}{r^{N+4}} ( -(N+2)^2 +(N+4)(N+2)- 2(N+2)) = 0,
\end{multline*}
or
\begin{equation*}
-2N\frac{c_1}{r^{N+2}} = 0, 
\end{equation*}
which implies $c_1 = 0$.

{\it Case 2.} When $ R < r <1,$ we have
\begin{align*}
-2N\frac{c_2}{r^{N+2}} = -2N\frac{\widetilde{c}}{r^{N+2}},
\end{align*}
which implies $c_2 = - \widetilde{c}$.
\\ 

Moreover, using \eqref{fed6} in \eqref{fed5}, we get\\

{\it Case 1.} When $ r < R$, we have
\begin{align*}
\alpha &\Big[ 2\Big(b_1 + \frac{d_1}{r^{N+2}}\Big)\delta_{kl} - N^2\frac{p_1}{r^{N+2}} + N(N+2)\frac{p_1}{r^{N+2}} + 2Nq_1 \Big] = - [\alpha - m + 2\alpha (\widetilde{b_1} -1)]\delta_{kl},\\
\mbox{or }\\
 & \alpha ( 2 d_1\delta_{kl} - N^2p_1 + N(N+2)p_1 )\frac{1}{r^{N+2}} + \alpha( 2b_1\delta_{kl} + 2N q_1 ) = - [ \alpha - m + 2\alpha (\widetilde{b_1} -1) ]\delta_{kl},
\end{align*}
which implies that
\begin{align}\label{fed71}
d_1\delta_{kl} + Np_1 &= 0 \\
\mbox{and}\ \ \  \alpha( 2b_1\delta_{kl} + 2N q_1 ) &= - [\alpha - m + 2\alpha (\widetilde{b_1} -1) ]\delta_{kl}.\label{fed72}
\end{align}
{\it Case 2.} When $R < r < 1$, we have
\begin{align*}
\beta \Big[ 2\big(b_2 + \frac{c_2}{r^N} + \frac{d_2}{r^{N+2}}\big)\delta_{kl} + 2N\frac{p_2}{r^{N+2}} + 2Nq_2 \Big] = - \Big[ \beta - m + 2\beta \big( \widetilde{b_2} -1 + \frac{\widetilde{c}}{r^N}\big) \Big]\delta_{kl},
\end{align*}
then, by using $c_2 = - \widetilde{c}$, it gives
\begin{align}\label{fed81}
 d_2\delta_{kl} + Np_2 &= 0 \\
\mbox{and}\ \  \beta( 2b_2\delta_{kl} + 2N q_2 ) &= - [ \beta - m + 2\beta (\widetilde{b_2} -1) ]\delta_{kl}.\label{fed82}
\end{align} 
\textbf{Step 3) Boundary Conditions: \\ 
i) Transmission conditions: }\\
\textit{a) Continuity of the $w_{kl}$ over the inner boundary at $r=R$:}
\begin{align}\label{fed91}
 b_1 + \frac{d_1}{R^{N+2}} &=\ b_2 + \frac{-\widetilde{c}}{R^N} + \frac{d_2}{R^{N+2}}, \\
 \Big(\frac{p_1}{R^N} + q_1R^2 + t_1\Big)\delta_{kl} &=\Big(\frac{p_2}{R^N} + q_2R^2 + t_2\Big)\delta_{kl}.\label{fed92}
\end{align}
\textit{b) Continuity of the flux over the inner boundary at $r=R$:}
We must rewrite the equation \eqref{auxN} in the following divergence form of 
\begin{align*}
-\frac{\partial}{\partial y_m}a(r)&\Big( \frac{\partial}{\partial y_m}w_{kl}(y) + \frac{1}{2}\big((w_{e_k}(y)-y_k)\delta_{lm} + (w_{e_l}(y) - y_l)\delta_{km}\big)\Big) \\
&=\ a(r)\delta_{kl} - m\delta_{kl} + \frac{1}{2}a(r)\Big( \frac{\partial}{\partial y_k} (w_{e_l}(y)-y_l) + \frac{\partial}{\partial y_l}(w_{e_k}(y)-y_k)\Big).    
\end{align*}
So, the boundary normal flux term, which we are concerned with, becomes\begin{align*}
&a(r)\Big( \frac{\partial}{\partial y_m}( w_{kl}(y) ) + \frac{1}{2}\big((w_{e_k}(y) - y_k)\delta_{km} + (w_{e_l}(y) - y_l)\delta_{km}\big)\Big)\cdot \nu \\
&= \ a(r)\Big( \frac{\partial}{\partial y_m}( y_ky_lg(r) +h(r) ) + \frac{1}{2}\big(y_k(f(r) - 1)\delta_{km} + y_l(f(r) - 1)\delta_{km}\big)\Big)\frac{y_m}{r} \\
&=\ y_ky_l\Big(  a(r)\Big( g^{\prime}(r)+ 2\frac{g(r)}{r} + \frac{f(r)-1}{r} \Big) \Big) + a(r)h^{\prime}(r). 
\end{align*}
Thus, from the required continuity of the boundary normal flux over the inner boundary at $r=R$, we get 
\begin{equation}\begin{aligned}\label{fed111}
\alpha \Big[ (b_1 + \frac{d_1}{r^{N+2}} )^{\prime} &+ 2\frac{(b_1 + \frac{d_1}{r^{N+2}})}{r} + \frac{(\widetilde{b_1}-1)}{r} \Big]|_{r=R} \\
&=\ \beta \Big[ \Big(b_2 + \frac{-\widetilde{c}}{r^N} + \frac{d_2}{r^{N+2}} \Big)^{\prime} + 2\frac{b_2 + \frac{-\widetilde{c}}{r^N} + \frac{d_2}{r^{N+2}}}{r} + \frac{\widetilde{b_2}+ \widetilde{-\widetilde{c}}{r^N} -1}{r} \Big]|_{r=R}\\
\end{aligned}\end{equation}
and
\begin{align}\label{fed112}
\alpha \Big( \frac{p_1}{r^N} + q_1r^2 +t_1 \Big)^{\prime}|_{r=R} \,\delta_{kl} =\ \beta \Big( \frac{p_2}{r^N} + q_2r^2 + t_2 \Big)^{\prime}|_{r=R}\, \delta_{kl}.
\end{align}
\textbf{ii) Dirichlet boundary condition:}
From the Dirichlet boundary condition of $w_{kl}$ on $\partial B(0,1),$ we get 
\begin{align}\label{fed101}
b_2 - \widetilde{c} + d_2  &= 0, \\
(p_2+ q_2 + t_2)\delta_{kl} &= 0. \label{fed102}
\end{align}
\textbf{Step 4)}
The unknown constants $(b_1,d_1),(p_1,q_1,r_1)$ and $(b_2,d_2),(p_2,q_2,t_2)$ can be found uniquely by solving equations \eqref{fed71} to \eqref{fed102}. There are $10$ unknown constants, $10$ linearly independent equations.  Here, $10$ coefficients
are uniquely determined, this confirms the already known fact, namely, unique solution to \eqref{auxN}. 
Now, we claim that the co-normal derivative of $w_{kl}$ on $\partial B(0,1)$ is zero, i.e.
\begin{equation}
\nabla w_{kl} (y)\cdot \nu = 0 \quad\mbox{on } \partial B(0,1).
\end{equation}
Above equation \eqref{fed12} is equivalent to two linear equations involving the coefficients $(b_2,d_2)$ and $(p_2,q_2,t_2)$
\begin{align}\label{fed131}
\Big[ (b_2 + \frac{-\widetilde{c}}{r^N} + \frac{d_2}{r^{N+2}} )^{\prime} + 2\frac{b_2 + \frac{-\widetilde{c}}{r^N} + \frac{d_2}{r^{N+2}}}{r}\Big]|_{r=1} &=0 \\
\mbox{and}
 \quad\quad\quad\quad\quad\quad\quad\quad\quad \quad \Big( \frac{p_2}{r^N} + q_2r^2 + t_2 \Big)^{\prime}\ |_{r=1}\ \delta_{kl} &=0.\label{fed132}
\end{align}
In order to establish our claim we have to show with the addition of these two new linear equations \eqref{fed131}, \eqref{fed132},
totally all these $12$ linear equations \eqref{fed71} to \eqref{fed132}
form a consistent system of $10$ unknown coefficients.
To this end, (for case of computation), we replace  \eqref{fed82}, \eqref{fed112} by 
\eqref{fed131}, \eqref{fed132} and we solve the resulting system of $10$ equations. Their solution
is then shown to satisfy \eqref{fed82}, \eqref{fed112} as well.\\
\\
First, we determine $d_2$ from \eqref{fed131} and consequently $p_2$ and $q_2$ from \eqref{fed81} and \eqref{fed132}, respectively, to get 
\begin{align}\label{fed141}
(N+2)d_2 - N\widetilde{c} =\ 0, &\quad\mbox{or}\ \ \ d_2 = \frac{N}{(N+2)}\widetilde{c},\\
 d_2\delta_{kl} + Np_2 =\ 0, &\quad\mbox{or}\ \ \ p_2 = - \frac{1}{(N+2)}\widetilde{c}\ \delta_{kl},\\
(-Np_2+ 2q_2) = 0, &\quad\mbox{or}\ \ q_2= -\frac{N}{2(N+2)}\widetilde{c}\ \delta_{kl}.
 \end{align}
Using them, we determine $b_2$ and $t_2$ from \eqref{fed101}, \eqref{fed102} respectively, to get
\begin{align}\label{fed15}
b_2 = \frac{2}{(N+2)}\widetilde{c},\quad t_2 =  \frac{1}{2}\widetilde{c}\ \delta_{kl} .
\end{align}
Next, we consider \eqref{fed91} and \eqref{fed111} to determine $(b_1, d_1)$ and  we get
\begin{equation*}\begin{aligned}
\alpha(N+2) b_1 =& \  ( \beta( \widetilde{b_2} + \frac{\widetilde{c}}{R^N} -1 ) - \alpha (\widetilde{b_1} -1 ) ) + \alpha\Big(\frac{2}{(N+2)} - \frac{1}{R^N} + \frac{N}{(N+2)}\frac{1}{R^{N+2}} \Big)\widetilde{c}\\
                 & +  \beta \Big( \frac{4}{(N+2)} + \frac{(N-2)}{R^N} - \frac{N^2}{(N+2)}\frac{1}{R^{N+2}} \Big)\widetilde{c}
\end{aligned}\end{equation*}
and
\begin{equation*}\begin{aligned}
\alpha(N+2)\frac{d_1}{R^{N+2}} =& -( \beta( \widetilde{b_2} + \frac{\widetilde{c}}{R^N} -1 ) - \alpha (\widetilde{b_1} -1 ) ) + 2\alpha\Big(\frac{2}{(N+2)} - \frac{1}{R^N} + \frac{N}{(N+2)}\frac{1}{R^{N+2}} \Big)\widetilde{c}\\
                                  &-  \beta \Big( \frac{4}{(N+2)} + \frac{(N-2)}{R^N} - \frac{N^2}{(N+2)}\frac{1}{R^{N+2}} \Big)\widetilde{c}.
\end{aligned}\end{equation*}
Successively, we can find $p_1$, $q_1 $ and $t_1$ from \eqref{fed71}, \eqref{fed72}
and \eqref{fed92}, respectively. \\
\\
Thus we have determined all $10$ coefficients and it is remained to check that the solutions obtained
above satisfies \eqref{fed82} and \eqref{fed112}. We recall \eqref{PN2} to write as
\begin{align*}
 \widetilde{b_2} -1 &=\ -\widetilde{c} \quad\mbox{and }\ \widetilde{b_1} -1 =\ \frac{(1-\theta)}{\theta}\widetilde{c}, \quad\mbox{where } \theta= R^N,\\
 \ m -\beta &=\ -N\beta\widetilde{c} \quad\mbox{and } \ m-\alpha = \ \frac{(1-\theta)((N-1)\beta +\alpha)}{\theta}\widetilde{c}.
\end{align*}
Now, let us see that LHS of \eqref{fed82} is equal to 
\begin{equation*}
\beta( 2b_2\delta_{kl} + 2N q_2 ) =\ \beta ( 2\frac{2}{(N+2)}\widetilde{c}\delta_{kl} - 2N\frac{N}{2(N+2)}\widetilde{c} ) = \beta (2-N)\widetilde{c}\delta_{kl}
\end{equation*}
and RHS of \eqref{fed82} is equal to
\begin{equation*}
-(\beta-m)\delta_{kl} - 2\beta(\widetilde{b_2} -1 )\delta_{kl} =\ -N\beta \widetilde{c}\delta_{kl} + 2\beta\widetilde{c}\delta_{kl}=\ \beta (2-N)\widetilde{c}\delta_{kl},
\end{equation*}
which is exactly  equal to the LHS of \eqref{fed82}.

Next,  we consider \eqref{fed112} and we have that:
\begin{align*}
\mbox{LHS of }&\eqref{fed112}:=\ \alpha\Big(\frac{p_1}{r^N} + q_1r^2 +t_1 \Big)^{\prime}\ |_{r=R}\ \delta_{kl} =\ \alpha\Big( -\frac{Np_1}{R^{N+1}} + 2q_1R \Big)\ =\ R\alpha\Big(\frac{d_1}{R^{N+2}} + 2q_1 \Big)\\
&=\ R\Big( \alpha\frac{d_1}{R^{N+2}} + \frac{-2\alpha b_1 - (\alpha -m) -2\alpha(\widetilde{b_1}-1)}{N}\delta_{kl} \Big)\ =\ R \beta \frac{N}{N+2} \Big(\frac{1}{R^{N+2}} -1 \Big)\widetilde{c}.\\
\mbox{RHS of }&\eqref{fed112}:=\ \beta \Big(\frac{p_2}{r^N} + q_2r^2 +t_2 \Big)^{\prime}\ |_{r=R}= \ \beta\Big( -\frac{Np_2}{R^{N+1}} + 2q_2R \Big)=\ R\beta \frac{N}{N+2} \Big(\frac{1}{R^{N+2}} -1 \Big)\widetilde{c}\\
              &\quad\quad\quad= \ \mbox{LHS of }\eqref{fed112}.
\end{align*}
Therefore, all these $12$ linear equations \eqref{fed71} to \eqref{fed112} form a consistent 
system for $10$ variables and which establishes our claim of 
having zero Neumann data together with zero Dirichlet data on the boundary of the unit ball. 
\hfill
\end{proof}
\paragraph{Resolution of \eqref{dl19}:}
Then, as we did before (see \eqref{ll6}), using $w_{kl}\in H^1_0(B(0,1))$ with $\nabla w_{kl}\cdot\nu = 0 $ on $\partial B(0,1)$,
we will define $\chi_{kl}\in H^1_{\#}(Y)$ in order to solve \eqref{dl19}. First, we define
$\widetilde{\chi_{kl} }\in H^1_{loc}(\mathbb{R}^N)$ by
\begin{align*}
\widetilde{\chi_{kl}}(y) = \varepsilon_{p}^2\ w_{kl}\Big(\frac{y-y^{p}}{\varepsilon_{p}}\Big)\ \ \mbox{ in }\ B(y^p,\varepsilon_p)\ \mbox{ a.e. in }\ \mathbb{R}^N,
\end{align*}
solving
\begin{equation}\label{auxN2} -div(a_{\mathbb{R}^N}(y)\nabla \widetilde{\chi_{kl}}(y)) = a_{\mathbb{R}^N}(y)\delta_{kl} - m\delta_{kl} - \frac{1}{2}\lb C_{l}(\chi_k) + C_{k}(\chi_l)\rb \ \mbox{ in }\mathbb{R}^N. \end{equation}
Then, we conclude that $\widetilde{\chi_{kl}}$ is a $Y$-periodic function and in order to get $\chi_{kl} \in H^1_{\#}(Y)$ with $\int_Y \chi_{kl} dy = 0 $ solving \eqref{auxN2} in $Y$,
we simply define $ \chi_{kl}(y) = \widetilde{\chi_{kl}}(y) - \frac{1}{|Y|}\int_Y \widetilde{\chi_{kl}} dy $, by restricting it in $Y$.
Subsequently, one has $X^{(2)}_Y= \eta_k\eta_l\chi_{kl}$, which solves the cell-problem \eqref{dl19}. 

% Thus we have defined 
% \begin{align*}
% X^{(1)}_{B(0,1)}(y) &=\ \eta_k( w_{e_k}(y) - y_k) \ \in H^1_0(B(0,1))\\
% X^{(1)}_Y(y) &= \ \varepsilon_p X^{(1)}_{B(0,1)}(\frac{y-y^p}{\varepsilon_p}) - \frac{1}{|Y|}\int_{B(y^p,\varepsilon_p)\cap Y}\varepsilon_p X^{(1)}_{B(0,1)}(\frac{y-y^p}{\varepsilon_p}) dy \ \mbox{ a.e. in }\underset{p\in K}{\cup}( Y\cap B(y^p,\varepsilon_p))\\                 
%            &=  \ \varepsilon_p X^{(1)}_{B(0,1)}(\frac{y-y^p}{\varepsilon_p}) - \frac{1}{|Y|}\sum_p \varepsilon_p^{N+1}\int_{B(0,1)}X^{(1)}_{B(0,1)}(y) dy \ \ \mbox{ in } B(y^p,\varepsilon_p)\ \mbox{ a.e. on } Y\\                 
%            &\ \in H^1_{\#}(Y).
% \end{align*}
% and
% \begin{align*}
% X^{(2)}_{B}(y) &=\ \eta_k\eta_l\hspace{2pt}w_{kl}(y) \ \in H^1_0(B(0,1)\quad\mbox{ and }\ \nabla X^{(2)}_B\cdot \nu = 0 \ \mbox{ on }\partial B(0,1).\\
% X^{(2)}(y) &= \ \varepsilon_p^2 X^{(2)}_{B}(\frac{y-y^p}{\varepsilon_p}) - \frac{1}{|Y|}\int_{\varepsilon_p B(0,1) +y^p}\varepsilon_p^2 X^{(2)}_{B}(\frac{y-y^p}{\varepsilon_p}) dy \ \mbox{in } \varepsilon_pB(0,1) +y^p; \mbox{ a.e. on } Y\\                 
%            &=  \ \varepsilon_p^2 X^{(2)}_{B}(\frac{y-y^p}{\varepsilon_p}) - \frac{1}{|Y|}\sum_p \varepsilon_p^{N+2}\int_{B(0,1)}X^{(2)}_{B}(y) dy \ \ \mbox{in } \varepsilon_pB(0,1) +y^p; \mbox{ a.e. on } Y.\\                 
%            &\ \in H^1_{\#}(Y).
% \end{align*}
\paragraph{Expression for $d_{PHS}$:}
Hence, the integral identity \eqref{dphs} 
of $d_{PHS}$ becomes
\begin{equation*}\begin{aligned}
-|Y|\cdot d_{PHS}\ \eta^4 &=\ \int_{Y} \mathcal{A}_Y\Big( X^{(2)}_Y -\frac{(X^{(1)}_Y)^2}{2}\Big)\cdot\Big( X^{(2)}_Y -\frac{(X^{(1)}_Y)^2}{2}\Big) dy\\
                  &= \sum_p \varepsilon_p^{N+2}\cdot \int_{B(0,1)} \mathcal{A}_{B(0,1)} \Big( X_{B(0,1)}^{(2)} - \frac{(X_{B(0,1)}^{(1)})^2}{2}\Big) \cdot\Big( X_{B(0,1)}^{(2)} - \frac{(X_{B(0,1)}^{(1)})^2}{2}\Big) dy.
\end{aligned}\end{equation*}
\textbf{Elimination of }$X_{B(0,1)}^{(2)}$:
We simplify the above expression to express it solely depending on $X^{(1)}_{B(0,1)}$ by eliminating $X^{(2)}_{B(0,1)}$.
Let us define 
$\widetilde{X}_{B(0,1)} = X^{(2)}_{B(0,1)} -\frac{(X^{(1)}_{B(0,1)})^2}{2}$, then  
using \eqref{ED7} and \eqref{auxN}, we write
\begin{equation}\label{phs3}
-div(a_B(y)\nabla \widetilde{X}_{B(0,1)}(y) )= div \Big( a_B(y) \nabla \frac{(X_{B(0,1)}^{(1)}(y) +\eta\cdot y)^2}{2} \Big) - \widetilde{m}\quad\mbox{ in }\ B.
\end{equation} 
It simply follows that 
\begin{equation*}\begin{aligned}
div\Big( a_B(y)\nabla\frac{(w_{e_k})^2}{2}\Big) &= \ div\big(a_B(y)\nabla(w_{e_k} - y_k)\big)\cdot(w_{e_k} - y_k) + a_B(y)\nabla(w_{e_k} - y_k)\cdot\nabla(w_{e_k} - y_k)\\
                           &= \ -div (a_B(y)e_k)\cdot(w_{e_k} - y_k) + a_B(y)\nabla(w_{e_k} - y_k)\cdot\nabla(w_{e_k} - y_k)\\
                           &=\ a_B(y)(\nabla(w_{e_k} - y_k) + e_k)\cdot(\nabla(w_{e_k} - y_k) + e_k) - a_B(y) + C^B_k(w_{e_k} - y_k) \\
                           &= \ div \Big( a_B(y) \nabla \frac{((w_{e_k} - y_k) + y_k)^2}{2} \Big) - a_B(y) + C^B_k(w_{e_k} - y_k).
\end{aligned}\end{equation*}
Now, multiplying  \eqref{phs3} by $\widetilde{X}_{B(0,1)}$, doing integration 
by parts and using  the fact that $X^{(1)}_{B(0,1)}= X^{(2)}_{B(0,1)}=0$ on $\partial B(0,1)$, we get
\begin{equation*}\begin{aligned}
|Y|\cdot d_{PHS}\ \eta^4 =  \ \sum_p \varepsilon_p^{N+2}\Bigg[\int_{B(0,1)} a_B(y)\nabla \frac{(X^{(1)}_{B(0,1)}+\eta\cdot y)^2}{2}&\cdot\nabla\Big( X^{(2)}_{B(0,1)} -\frac{(X^{(1)}_{B(0,1)})^2}{2}\Big) \\                                    - &\widetilde{m}\int_{B(0,1)} \frac{(X^{(1)}_{B(0,1)})^2}{2} dy\Bigg].
\end{aligned}\end{equation*}
Then, multiplying \eqref{phs3} by $\frac{(X^{(1)}_{B(0,1)}+\eta\cdot y)^2}{2}$, doing the integration by parts and using the fact that 
$X^{(1)}_{B(0,1)}=0$ and $\nabla X^{(2)}_{B(0,1)}\cdot\nu=0$ on $\partial B(0,1)$, we finally get
\begin{equation}\label{phsd}\begin{aligned}
-d_{PHS}\ \eta^4\cdot|Y| = \ \sum_p \varepsilon_p^{N+2}\Bigg[\int_{B(0,1)} a_B(y)&\nabla \frac{(X^{(1)}_{B(0,1)}+\eta\cdot y)^2}{2}\cdot \nabla \frac{(X^{(1)}_{B(0,1)}+\eta\cdot y)^2}{2} dy\\
                                                           + &\widetilde{m}\int_{B(0,1)}\frac{(X^{(1)}_{B(0,1)})^2}{2} dy\Bigg].
\end{aligned}\end{equation}
\textbf{Final expression for $d_{PHS}$:}
Moreover, due to explicit formula of the solution $X^{(1)}_{B(0,1)}= \eta_k(w_{e_k}(y)-y_k)\ $ (see \eqref{ED5}) and taking $\eta = e_k$, we express \eqref{phsd} as follows:
\begin{equation}\begin{aligned}\label{frN}
-2|Y|\cdot d_{PHS}=\ \sum_p &\varepsilon_p^{N+2} \cdot \ \Bigg[ \int_{B(0,R)}{\widetilde{b_1}}^2(m-\alpha {\widetilde{b_1}}^2) y_k^2\ dy \\
 & + \int_{B(0,1)\smallsetminus B(0,R)} \Big( m-\beta\big( \widetilde{b_2} + \frac{\widetilde{c}}{r^N}\big)^2 \Big|\nabla \Big(y_k\big(\widetilde{b_2} +\frac{\widetilde{c}}{r^N}\big)\Big) \Big|^2 \Big) y_k^2\ dy \Bigg], 
\end{aligned}\end{equation} 
where $\widetilde{b_1}$, $\widetilde{b_2}$, $\widetilde{c}$ are known in terms of given data $\alpha, \beta, \theta$ and $N$ found in \eqref{PN2}.
%or,
%\begin{equation}
%d_{PHS}=  \ \frac{1}{|Y|}\int_{Y} \frac{(X^{(1)}+\eta.y)^2}{2} \lb A\nabla(X^{(1)} +\eta.y).\nabla (X^{(1)}+\eta.y) dy -  \widetilde{M}\rb dy.
%\end{equation}
%The \eqref{phsd} solely depends on $X^{(1)}_B$.  
%Let us consider the expression \eqref{phsd} and considering the standard basis vector $e$ we write 
%\begin{align}\label{conb}
%|Y|\cdot d_{PHS} =&\ \int_{Y} \frac{(\chi_e + (e,y))^2}{2} a(y)|\nabla (\chi_e + (e,y))|^2 dy -  \int_{Y} M  \frac{(\chi_e + (e,y))^2}{2} dy\\
%=&\ \sum_k \{ \int_{\varepsilon_kB(0,1) +y^{k}} \frac{\lb \varepsilon_k w_{e}(\frac{y-y^{k}}{\varepsilon_k})\rb^2}{2} a(\frac{y-y^{k}}{\varepsilon_k})|\nabla w_e(\frac{y-y^{k}}{\varepsilon_k})|^2 dy \\
% &\ \ \quad\quad\quad \quad\quad\quad\quad \quad\quad\quad \quad \quad -  \int_{\varepsilon_kB(0,1) +y^k} M  \frac{\lb\varepsilon_k w_e(\frac{y-y^{k}}{\varepsilon_k})\rb^2}{2} dy \}\\
%\end{align}
\noindent
\begin{remark}
It is already known from the expression \eqref{dphs} (or \eqref{phsd})
that the dispersion tensor $d_{PHS}$ is a non-positive definite tensor. 
%and here the expression \eqref{phsd} shows that it can be also decomposed into two positive terms involving only the $X^{(1)}$ variable;
Moreover, the expression \eqref{phsd} tells us that $d_{PHS}$ depends upon only 
on the scales $\{ \varepsilon_p \}_{p\in K}$, not on the translations $\{ y^p \}_{p\in K}$. 
The last aspect
signifies the following property of $d_{PHS}$: position of balls in a
micro-structure is not important; what matters is only their radii. Thus, two
Hashin-Shtrikman micro-structures consisting of different arrangements of 
core-coating balls, but having the same set of radii, have the same value of 
the Bloch dispersion coefficient $d_{PHS}$. This aspect of Hashin-Shtrikman micro-structure
is not obvious to start with. Thus, even though $d_{PHS}$ varies among Hashin-Shtrikman
structures, its variation is somewhat
special: it does not depend on position of centers of balls in the micro-structure;
it depends only on radii. Recall that the homogenized coefficient $m$ depends only on $\{\alpha,\beta,\theta,N\}$,
but independent of radii $\{\varepsilon_p\}_{p\in K}$ and centers $\{y^p\}_{p\in K}$.
\end{remark}

The above computation reduced the original problem which was posed on micro-structures, to space 
of sequences $l^1$ (cf. `Conjecture' at the end of the Section 3). As long as the macro quantity $(\widetilde{b_1}, \widetilde{b_2},\widetilde{c})$
are getting fixed through \eqref{PN2}, from the expression \eqref{frN} we can compute
$d_{PHS}$ explicitly and moreover, due to the negativity of the dispersion tensor, it
will be maximized or minimized whenever $\underset{p}{\sum} \varepsilon_p^{N+2}$ is 
got minimized or maximized, respectively under the constraint 
$\underset{p}{\sum} \varepsilon_p^N = c_N$ (a dimension constant).
As a next step, we exploit the properties of $l^1$ to prove existence of minimizers.  

\section{Proof of the Conjecture}
\setcounter{equation}{0}

Let us consider a Vitali covering of $Y=[0,1]^N$ with a countable infinite union 
of disjoint balls with center $y^p$ and radius $\varepsilon_p$, where
$p\in K$, i.e. $meas\big(Y\smallsetminus  \underset{p\in K}{\cup} (B(y^p,\varepsilon_p)\cap Y)\big)=0$.
%Let us consider a Vitali covering of $Y$ by the reduced copy of $B(0,1)$. We consider the 
%sequences radius $\{\varepsilon_{p}\}_{p\in\mathbb{N}}$, and centers $\{y^{p}\}_{p\in\mathbb{N}}$ such that  $|Y \smallsetminus \underset{p\in K}{\cup}(\varepsilon_{p}B(0,1) + y^p)| = 0$,  
%where the sets $\varepsilon_{p}B(0,1) + y^p, p\in K$ are disjoint. 
We first rearrange the sequence $\{\varepsilon_{p}\}_{p\in K}$ to make it as a decreasing sequence, 
i.e., $\varepsilon_1 \geq \varepsilon_2 \geq \ldots \geq \varepsilon_p \geq \ldots $ Let us define $d_p = \frac{\varepsilon_p^N}{c_N} $, with  $d_1 \geq d_2 \geq \ldots \geq d_p \geq \ldots $.
Then, we want to minimize $I = -c_N^{\frac{N+2}{N}}\underset{p}{\sum} d_p^{\frac{N+2}{N}} $,
under the constraint $\underset{p}{\sum} d_p = 1$:
\begin{center}
\textit{Minimization of $I = -c_N^{\frac{N+2}{N}}\underset{p}{\sum} d_p^{\frac{N+2}{N}} $ under the constraint $\underset{p}{\sum}d_p = 1$. }
\end{center}
\noindent
\textbf{Difficulties with the minimization problem:}\\
The peculiarity of our problem is that it is concerned with a (constrained) minimization of a strictly
concave functional over the unit sphere of $l^1$ (the unit sphere
representing the constraint set). From the point of view of Functional
Analysis, difficulties with the existence of a minimizer are
well-known, owing to the non-reflexivity of $l^1$. In general, working in such
spaces, a bounded sequence may not have any weakly converging subsequence
and even if it has one, we cannot conclude it satisfies the
constraint.\\
\\
\textbf{Existence of Minimizers:}\\
Fortunately, in our case, the criterion of de la Vallee Pousin \cite[Page no. 19]{MP}
is applicable and it guarantees $l^1$-weak compactness of a
minimizing sequence. It is also known that weak and norm convergences are
equivalent in the case of $l^1$. These arguments establish that any
minimizing sequence has a converging subsequence in $l^1$ and, so, a minimizer
satisfying the constraint exists.\\
\\
\textbf{Uniqueness Issue:}\\
However, uniqueness is an issue since we have a strictly
concave functional to minimize. Uniqueness of the minimizer can however be
proved using other arguments  below. Combining both  results, we
obtain that the entire minimizing sequence is $l^1$ strongly convergent.
It is not clear how uniqueness can be proved by analytical method.
However, with a geometric point of view, we can settle both existence and
uniqueness of minimizer. This is done in the sequel:\begin{figure}
 \begin{center}
  \includegraphics[width = 6cm]{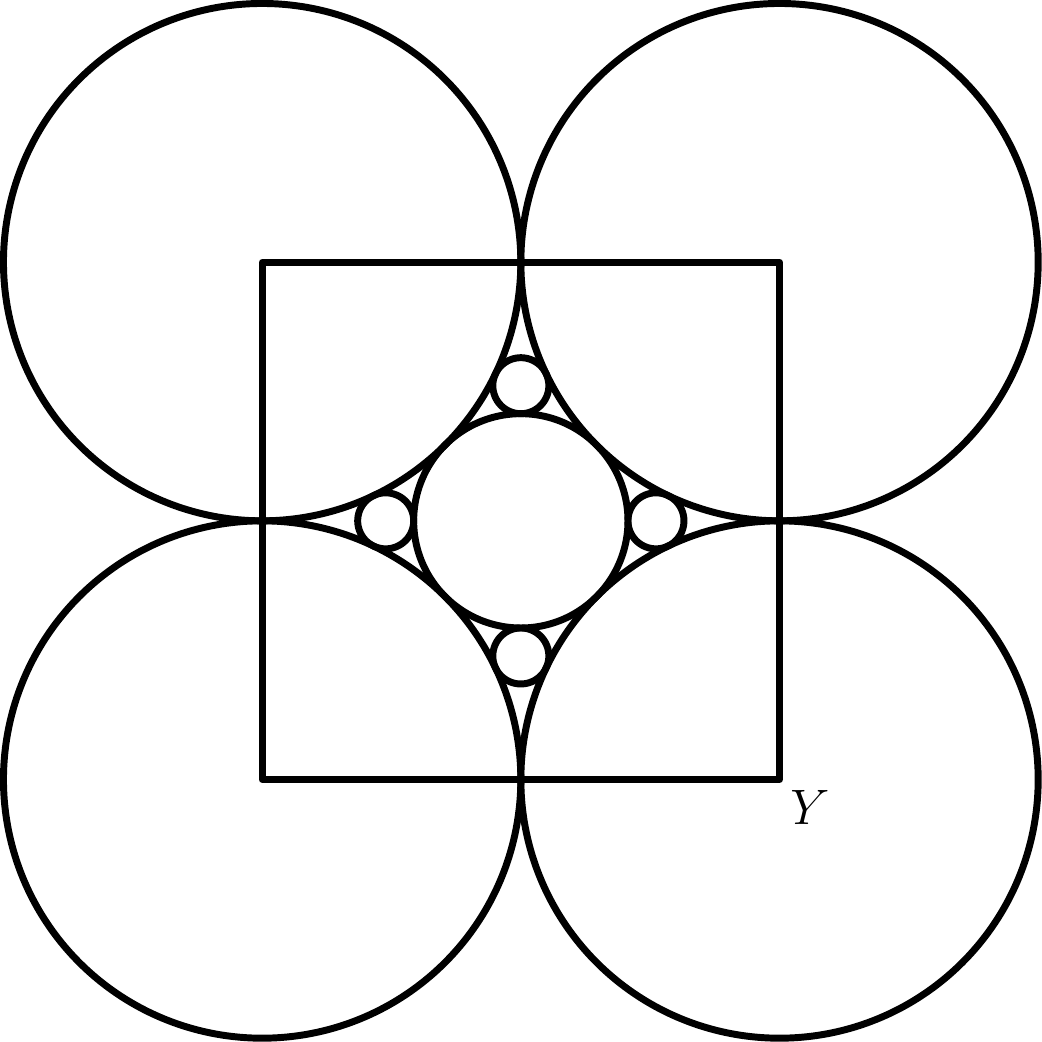}
  \caption{Apollonian-Hashin-Shtrikman structures in $Y$.}
 \end{center}
\end{figure}
\paragraph{Existence and Uniqueness via Geometrical Method:}
\noindent
For a geometrical picture, we ask the reader to imagine the flat torus 
obtained by identifying the opposite sides of the cell $Y=[0,1]^N$. 
We have already seen that the dispersion coefficient $d_{PHS}$ is 
invariant under translation, and as $Y$ is identified with $\mathbb{R}^N$ through $\mathbb{Z}^N-$translation invariance.
Then, we first find a Hashin-Shtrikman construction to cover the whole space $\mathbb{R}^N$ and if it is invariant 
under $\mathbb{Z}^N-$translations, then we will consider it as a Hashin-Shtrikman structure for $Y$ and conversely.
Note that it is enough to consider decreasing sequences of 
non-negative numbers in the minimization process. These numbers represent 
the radii of balls in the Hashin-Shtrikman micro-structure.
Finding the first element of the minimizer, being the
highest, amounts to putting a ball with maximum radius $\big(\varepsilon_1^{*} = \frac{1}{2}\big)$
inside the torus. It is geometrically clear that this ball (and hence its radius) is uniquely
determined. In the second step, the same pattern is repeated: the second
element of the minimizer, being the next highest, represents the radius $\big(\varepsilon_2^{*} = \frac{\sqrt{2}-1}{2}\big)$
of the biggest ball embedded in the complement of the previous
ball. Again, this is unique. In the third step, we observe that there
is no uniqueness and in fact there are four balls of maximum radii
$\big(\varepsilon_3^{*} = \frac{(\sqrt{2}-1)(2\sqrt{2}-1)}{14}\big)$,
which can be placed in the complement of the union of the first and the second
balls. The radii of these four balls are however equal. This amount to
saying that the third, fourth, fifth, sixth elements of the minimizer are
equal. The above argument can be repeated at every subsequent step and this
procedure identifies the \textit{Apollonian-Hashin-Shtrikman} micro-structures (Figure 2) 
as the unique solution of our geometric problem. The radii $\{\varepsilon_p^{*}\}_{p\in K}$ of the balls thus 
obtained provide the minimizer for our constrained minimization problem. 
Hence, we denote the minimum value as $I_{min}= -c_N^{\frac{N+2}{N}} \underset{p}{\sum} (\varepsilon_p^{*})^{N+2}$.    
\paragraph{Optimal bounds on $I$:}
We have found the minimum value of $I$ with its unique minimizer, next we find its maximum value under the constraint $\sum_p d_p =1$. We simply see that
\begin{center}
$I = -\ c_N^{\frac{N+2}{N}}\underset{p}{\sum} d_p^{\frac{N+2}{N}} \leq -c_N^{\frac{N+2}{N}}d_1^{\frac{2}{N}}. $
\end{center}
Clearly, $d_1>0$ can be chosen arbitrarily small. Thus, $0$ is the supremum value of $I$
and it is not the maximum value of $I$. So, unlike to the previous case of minimization, here the maximizer doesn't 
exist in the classical micro-structures. In particular, we have a bound for $I$, that is, $I \in [I_{min},0)$. 
\hfill
\qed

\subsection*{Acknowledgments} The authors thank support from ECOS-CONICYT Grant~C13E05.
The second author is partially supported by PFBasal-001 and PFBasal-003 projects and by Fondecyt Grant~N$^\circ$1140773. PFBasal-001 also partially supports the fourth author.

\bibliographystyle{plain}
\bibliography{Master_bibfile}
\end{document}